\documentclass[a4paper,12pt]{amsart}

\usepackage{amsmath}
\usepackage{amsthm}
\usepackage{amssymb}
\usepackage{amscd}
\usepackage{mathrsfs}
\usepackage[all]{xy}
\usepackage{enumerate}

\usepackage{color}

\numberwithin{equation}{section}

\theoremstyle{plain}
\newtheorem{thm}{Theorem}[section]
\newtheorem{dfn}[thm]{Definition}
\newtheorem{prop}[thm]{Proposition}
\newtheorem{cor}[thm]{Corollary}
\newtheorem{lem}[thm]{Lemma}
\newtheorem*{thm*}{Theorem}
\newtheorem{conj}[thm]{Conjecture}
\newtheorem*{cl*}{Claim}

\theoremstyle{definition}
\newtheorem{exa}[thm]{Example}
\newtheorem{rem}[thm]{Remark}

\newcommand{\nc}{\newcommand}

\nc{\mf}[1]{{\mathfrak{#1}}}
\nc{\mb}[1]{{\mathbf{#1}}}
\nc{\bb}[1]{{\mathbb{#1}}}
\nc{\ms}[1]{{\mathscr{#1}}}
\nc{\cal}[1]{{\mathcal{#1}}}
\nc{\ol}[1]{{\overline{#1}}}
\nc{\E}[1]{{\mb{E}_{\mb{#1}}}}
\nc{\pE}[1]{{\mb{E}'_{{#1}}}}
\nc{\ppE}[1]{{\mb{E}''_{{#1}}}}
\nc{\EO}[2]{{\mb{E}_{\mb{#1},#2}}}
\nc{\pEO}[2]{{\mb{E}'_{\mb{#1},#2}}}
\nc{\ppEO}[2]{{\mb{E}''_{\mb{#1},#2}}}
\nc{\EE}[2]{{\mb{E}(\mb{#1})_{#2}}}
\nc{\Fl}[2]{{\cal{F}_{\mb{#1},\mb{#2}}}}
\nc{\tFl}[3]{{\widetilde{\cal{F}}_{\mb{#1},\mb{#2};{#3}}}}
\nc{\LL}[3]{{L_{\mb{#1},\mb{#2};{#3}}}}
\nc{\mP}[2]{{\ms{P}_{\mb{#1},{#2}}}}
\nc{\mQ}[2]{{\ms{Q}_{\mb{#1},{#2}}}}

\nc{\tE}[1]{{{}^\theta\mb{E}_{\mb{#1}}}}
\nc{\tpE}[1]{{{}^\theta\mb{E}'_{{#1}}}}
\nc{\tppE}[1]{{{}^\theta\mb{E}''_{{#1}}}}
\nc{\tEO}[2]{{{}^\theta\mb{E}_{\mb{#1},#2}}}
\nc{\tpEO}[2]{{{}^\theta\mb{E}'_{\mb{#1},#2}}}
\nc{\tppEO}[2]{{{}^\theta\mb{E}''_{\mb{#1},#2}}}
\nc{\tEE}[2]{{{}^\theta\mb{E}(\mb{#1})_{#2}}}
\nc{\thF}[2]{{{}^\theta\cal{F}_{\mb{#1},\mb{#2}}}}
\nc{\ttFl}[3]{{{}^\theta\widetilde{\cal{F}}_{\mb{#1},\mb{#2};{#3}}}}
\nc{\tLL}[3]{{{}^{\theta}\!L_{\mb{#1},\mb{#2};{#3}}}}

\nc{\tmP}[2]{{{}^\theta\!\!\ms{P}_{\mb{#1},{#2}}}}
\nc{\tmQ}[2]{{{}^\theta\!\!\ms{Q}_{\mb{#1},{#2}}}}
\nc{\tpi}[2]{{{}^\theta\pi_{\mb{#1},\mb{#2}}}}
\nc{\tK}[2]{{}^{\theta}\!K_{\mb{#1},{#2}}}
\nc{\ttK}[1]{{}^{\theta}\!K_{\mb{#1}}}
\nc{\tKK}[1]{{}^{\theta}\!K_{#1}}

\nc{\cl}{\colon}
\nc{\scbul}{{\,\raise1pt\hbox{$\scriptscriptstyle\bullet$}\,}}
\nc{\vac}{{\phi}}
\nc{\te}{\tilde{e}}
\nc{\tei}{\tilde{e}_i}
\nc{\tf}{\tilde{f}}
\nc{\tfi}{\tilde{f}_i}
\nc{\tEi}{\tilde{E}_i}
\nc{\tFi}{\tilde{F}_i}
\nc{\isoto}[1][]%
{{\mathop{\buildrel{\sim}\over\longrightarrow}\limits_{#1}}}
\nc{\low}{{\mathrm{low}}}
\nc{\upper}{{\mathrm{up}}}
\nc{\seteq}{\mathbin{:=}}
\nc{\set}[2]{\left\{#1 \mid #2 \right\}}
\nc{\Uf}[1][\mf{g}]{U^-_v(#1)}
\nc{\U}[1][\mf{g}]{U_v(#1)}
\nc{\tV}{{}^\theta\cal{V}}
\nc{\tG}[1]{{}^\theta\mb{G}_{\mb{#1}}}

\nc{\eq}{\begin{eqnarray}}
\nc{\eneq}{\end{eqnarray}}
\nc{\eqn}{\begin{eqnarray*}}
\nc{\eneqn}{\end{eqnarray*}}
\nc{\ba}{\begin{array}}
\nc{\ea}{\end{array}}
\nc{\benu}{\begin{enumerate}}
\nc{\eenu}{\end{enumerate}}
\newcommand{\bnum}{\begin{enumerate}}
\newcommand{\enum}{\end{enumerate}}
\nc{\bit}{\begin{itemize}}
\nc{\eit}{\end{itemize}}

\def\vo{\mathop{\mathrm{out}}\nolimits}
\def\vi{\mathop{\mathrm{in}}\nolimits}
\def\ud{\mathop{\mathrm{\wt}}\nolimits}
\def\Hom{\mathop{\mathrm{Hom}}\nolimits}

\def\Res{\mathop{\mathrm{Res}}\nolimits}
\def\Ad{\mathop{\mathrm{Ad}}\nolimits}
\def\rHom{\mathop{\mb{R}\mathscr{H}\!om}\nolimits}
\def\al{\mathop{\mathrm{\alpha}}\nolimits}
\def\la{\mathop{\mathrm{\lambda}}\nolimits}
\def\wt{\mathop{\mathrm{wt}}\nolimits}
\def\ord{\mathop{\mathrm{ord}}\nolimits}
\def\id{\mathop{\mathrm{id}}\nolimits}
\def\Coker{\mathop{\mathrm{Coker}}\nolimits}

\def\Supp{\mathop{\mathrm{Supp}}\nolimits}
\def\Alt{\mathop{\mathrm{Alt}}\nolimits}
\def\pt{\mathop{\mathrm{pt}}\nolimits}
\def\rank{\mathop{\mathrm{rank}}\nolimits}
\def\Ig{\mathop{\mathrm{Im}}\nolimits}
\def\Perv{\mathop{\mathrm{Perv}}\nolimits}
\def\tr{\mathop{\mathrm{tr}}\nolimits}
\def\IC{\mathop{\mathrm{IC}}\nolimits}

\title{A quiver construction of symmetric crystals}

\author{Naoya Enomoto}
\address{Research Institute for Mathematical Sciences\\
Kyoto University\\
Kyoto 606--8502, Japan
}
\email{henon@kurims.kyoto-u.ac.jp}

\thanks{The author is partially supported by JSPS Research Fellowships for Young Scientists.}

\begin{document}
\maketitle
\begin{abstract}
In the papers \cite{EK}, \cite{EK2} and \cite{EK3} with Masaki Kashiwara, the author introduced the notion of symmetric crystals and presented the Lascoux-Leclerc-Thibon-Ariki type conjectures for the affine Hecke algebras of type $B$. Namely, we conjectured that certain composition multiplicities and branching rules for the affine Hecke algebras of type $B$ are described by using the lower global basis of symmetric crystals of $V_\theta(\lambda)$. In the present paper, we prove the existence of crystal bases and global bases of $V_\theta(0)$ for any symmetric quantized Kac-Moody algebra by using a geometry of quivers (with a Dynkin diagram involution). This is analogous to George Lusztig's geometric construction of $U_v^-$ and its lower global basis.
\end{abstract}
\section{Introduction}
\subsection{}
\ Let $K^{\rm{AHA}}_n$ be the Grothendieck group of the affine Hecke algebra $H_n(q)$ of type $A_n$ and set $K^{\rm{AHA}}=\bigoplus_{n \ge 0}K_n^{\rm{AHA}}$. Generalizing the LLT conjecture \cite{LLT} for the Hecke algebra of type $A$, S. Ariki \cite{Ari} proved that $K^{\rm{AHA}} \otimes_{\bb{Z}} \bb{C}$ is isomorphic to $U^{-}(\mf{g})$ as $U^-(\mf{g})$-modules. Here $\mf{g}=\widehat{\mf{sl}}_{\ell-1}$ or $\mf{gl}_\infty$ according that the parameter $q$ of the affine Hecke algebras of type $A$ is a primitive $\ell$-th root of unity or not a root of unity. This isomorphism sends the irreducible modules of the affine Hecke algebras to the specialization of the upper global basis of $U_v^-(\mf{g})$ at $v=1$. His proof is based on two results in the geometric representation theory. One is the equivariant $K$-theoretic description of the irreducible and standard modules of the affine Hecke algebras by Chriss-Ginzburg and Kazhdan-Lusztig, and the other is G. Lusztig's geometric construction \cite{Lus1} of the lower global basis of $U_v^-(\mf{g})$. Lusztig's theory is summarized as follows. \\
\quad Let $\mf{g}$ be a symmetric Kac-Moody algebra and $I$ an index set of simple roots of $\mf{g}$. For a fixed set of arrows $\Omega$, we consider $(I,\Omega)$ as a (finite) oriented graph. We call $(I,\Omega)$ a quiver. For an $I$-graded vector space $\mb{V}$, we define the moduli space of representations of quiver $(I,\Omega)$ by
\[
\EO{V}{\Omega}=\bigoplus_{i \stackrel{\Omega}{\longrightarrow} j}\Hom(\mb{V}_i,\mb{V}_j).
\]
The algebraic group $G_{\mb{V}}=\prod_{i \in I}GL(\mb{V}_i)$ acts on $\EO{V}{\Omega}$. Lusztig introduced a certain full subcategory $\mQ{V}{\Omega}$ of $\ms{D}(\EO{V}{\Omega})$ where $\ms{D}(\EO{V}{\Omega})$ is the bounded derived category of constructible complexes of sheaves on $\EO{V}{\Omega}$ (for the definition, see section \ref{sec:Lus}). Let $K(\mQ{V}{\Omega})$ be the Grothendieck group of $\ms{Q}_{\mb{V},\Omega}$. He constructed the induction operators $f_i$ and the restriction operators $e'_i$ on the Grothendieck group $K_{\Omega}:=\oplus_{\mb{V}}K(\mQ{V}{\Omega})$, where $\mb{V}$ runs over the isomorphism classes of $I$-graded vector spaces. He proved the following theorem.
\begin{thm}[Lusztig]\label{Lthm1} \quad
\bnum[(i)]
\item The operators $e_i'$ and $f_i$ define the action of the reduced $v$-analogue $B_v(\mf{g})$ of $\mf{g}$ on $K_{\Omega} \otimes_{\bb{Z}[v,v^{-1}]} \bb{Q}(v)$, and $K_{\Omega} \otimes_{\bb{Z}[v,v^{-1}]} \bb{Q}(v)$ is isomorphic to $U_v^-(\mf{g})$ as a $B_v(\mf{g})$-module. The involution induced by the Verdier duality functor coincides with the bar involution on $U_v^-(\mf{g})$.
\item The simple perverse sheaves in $\bigoplus_{\mb{V}}\ms{Q}_{\mb{V},\Omega}$ give the lower global basis of $U_v^-(\mf{g})$.
\enum
\end{thm}
\subsection{}
\ Recently in \cite{EK} and \cite{EK2} with M. Kashiwara, the author presented an analogue of the LLTA conjecture for the affine Hecke algebra of {\em type $B$}. In \cite{EK2}, we considered $U_v(\mf{g})$ and its Dynkin diagram involution $\theta$ and constructed an analogue $B_\theta(\mf{g})$ of the reduced $v$-analogue $B_v(\mf{g})$ (for the definition, see Definition \ref{def:Bt} below). We gave a $B_\theta(\mf{g})$-module $V_\theta(\lambda)$ for a dominant integral weight $\lambda$ such that $\theta(\lambda)=\lambda$, which is an analogue of the $B_v(\mf{g})$-module $U_v^-(\mf{g})$ (for the definition, see Definition \ref{prop:Vtheta} below). We defined the notion of symmetric crystals and conjectured the existence of the global basis. In the case $\mf{g}=\mf{gl}_\infty,$ $I=\bb{Z}_{\text{odd}}$, $\theta(i)=-i$ and $\lambda=0$, we constructed the PBW type basis and the lower (and upper) global basis parametrized by the $\theta$-restricted multi-segments. We conjectured that irreducible modules of the affine Hecke algebras of type $B$ are described by the global basis associated to the symmetric crystals. 
\subsection{}
\ In this paper, we construct the lower global basis for the symmetric crystals by using a geometry of quivers (with a Dynkin diagram involution). Hence for any symmetric quantized Kac-Moody algebra $U_v(\mf{g})$, we establish the existence of a crystal basis and a global basis for $V_\theta(0)$. \\
\quad We introduce the notion of $\theta$-quivers. This is a quiver $(I,\Omega)$  with an involution $\theta:I \to I$ (and $\theta:\Omega \to \Omega$) satisfing some conditions (see Definition \ref{def-tq}). This notion is partially motivated by Syu Kato's construction \cite{Kt} of the irreducible representations of the affine Hecke algebras of type $B$. \\
\quad We also introduce the $\theta$-symmetric $I$-graded vector spaces. This is an $I$-graded vector space $\mb{V}=(\mb{V}_i)_{i \in I}$ endowed with a non-degenerate symmetric bilinear form such that $\mb{V}_i$ and $\mb{V}_j$ are orthogonal if $j \neq \theta(i)$. For a $\theta$-quiver $(I,\Omega)$ and a $\theta$-symmetric $I$-graded vector space $\mb{V}$, we define the moduli space $\tEO{V}{\Omega}$ of representations of $(I,\Omega)$ adding a skew-symmetric condition on $\EO{V}{\Omega}$ with respect to the involution $\theta$. \\
\quad Similarly to Lusztig's arguments, we consider a certain full subcategory $\tmQ{V}{\Omega}$ of $\ms{D}(\tEO{V}{\Omega})$ and its Grothendieck group $\tK{V}{\Omega}$. We define the induction operators $F_i$ and the restriction operators $E_i$ on $\tKK{\Omega}:=\oplus_{\mb{V}}\tK{V}{\Omega}$ where $\mb{V}$ runs over the isomorphism classes of the $\theta$-symmetric $I$-graded vector spaces. We prove the following main theorem which is an analogous result of Theorem \ref{Lthm1}.
\begin{thm}[Theorem \ref{mt2}]\label{mthm2}
$\tKK{\Omega} \otimes_{\bb{Z}[v,v^{-1}]} \bb{Q}(v) \cong V_\theta(0)$ as $B_\theta(\mf{g})$-modules. The simple perverse sheaves in $\tKK{\Omega}$ give a lower global basis of $V_\theta(0)$.
\end{thm}
Though Lusztig proved Theorem \ref{Lthm1} using some inner product on $K_{\Omega}$, we prove Theorem \ref{mthm2} using a criterion of crystals (Theorem \ref{cricry}) and certain estimates for the actions of $E_i$ and $F_i$ on simple perverse sheaves (Theorem \ref{est}).\\
\quad This paper is organized as follows. \\
\quad In section 2, we recall some results on the quantum enveloping algebras, the theory of the crystal bases and the global bases, the notion of symmetric crystals and known results of perverse sheaves and Fourier-Sato-Deligne transforms. Especially, we recall an important criterion of crystals in Theorem \ref{cricry}. We use this in our proof of existence of the crystal basis of $V_\theta(0)$. \\
\quad In section 3, we give a quick review on Lusztig's construction of $U_v^-(\mf{g})$ and its lower global basis. \\
\quad In section 4, we introduce the notion of $\theta$-quivers and $\theta$-symmetric $I$-graded vector spaces. We define the category $\tmQ{V}{\Omega}$ and the induction operators $F_i$ and the restriction operators $E_i$. We calculate actions of $E_i$ and $F_i$ on $\tmQ{V}{\Omega}$. We also prove that $E_i$ and $F_i$ commute with the Fourier-Sato-Deligne transforms. \\
\quad In section 5, we introduce the Grothendieck group $\tKK{\Omega}$ and show three key results. First, we calculate the commutation relations of $E_i$ and $F_i$. Second, we give certain estimates of coefficients with respect to the action of $E_i$ and $F_i$ on simple perverse sheaves. These estimates satisfy the condition in Theorem \ref{cricry}. Third, we prove the invariance of simple perverse sheaves with respect to the Verdier duality functor. Combining these results we prove the main theorem. 

\begin{rem}
We give two remarks on a difference from the "folding" procedure and an overlap with perverse sheaves arising from graded Lie algebras by Lusztig.
\begin{enumerate}[(i)]
\item Our construction is completely different from Lusztig's construction, "Quiver with automorphisms", in his book \cite[Chapter.12-14]{Lus3}. \\
\quad He considered actions $a:I \to I$ and $a:H \to H$ induced from a finite cyclic group $C$ generated by $a$. Put an orientation $\Omega$ such that $\vo(a(h))=a(\vo(h))$ and $\vi(a(h))=a(\vi(h))$. He said this orientation "compatible". Let $\mathcal{V}^a$ be the category of $I$-graded vector spaces $\mb{V}$ such that $\dim\mb{V}_i=\dim\mb{V}_{a(i)}$ for any $i \in I$. For $\mb{V} \in \mathcal{V}^a$, $a$ induces a natural automorphism on $\EO{V}{\Omega}$ and a functor $a^*:\ms{D}(\EO{V}{\Omega}) \to \ms{D}(\EO{V}{\Omega})$. He introduced "$C$-equivariant" simple perverse sheaves $(B,\phi)$, where $B$ is a perverse sheaf on $\EO{V}{\Omega}$ and $\phi:a^*B \cong B$. Then he proved that the set $\sqcup_{\mb{V} \in \mathcal{V}^a}\mb{B}_{\mb{V},\Omega}$ of $C$-equivariant perverse sheaves gives a lower global basis of $\mb{U}_v^-(\mf{g})$. Here $\mf{g}$ has a non-symmetric Cartan matrix which is obtained by the "folding" procedure with respect to the $C$-action on $I$.\\
\quad But in our construction, a $\theta$-orientation is not a compatible orientation. Moreover the most essential difference is that his construction has no skew-symmetric condition in our sence. Hence the set of simple perverse sheaves ${}^\theta\!\ms{P}_{\mb{V},\Omega}$ and the space ${}^\theta\!K_{\Omega} \otimes_{\bb{Z}[v,v^{-1}]}\bb{Q}(v) \cong V_\theta(0)$ are different from $\mb{B}_{\mb{V},\Omega}$ and $\mb{U}_v^-(\mf{g})$, respectively. The explicit crystal structure of $V_\theta(0)$ is unknown except for the case $\mf{g}=\mf{gl}_\infty$, $I=\bb{Z}_{\rm{odd}}$ and $\theta(i)=-i$ in \cite{EK2}.
\item In some special case, the lower global basis constructed in this paper is obtained by Lusztig (\cite{Lus4} and \cite{Lus5}). Let us consider the case $G=SO(2n,\bb{C})$. Let $\mf{g}$ be the Lie algebra of $G$ and $T$ a fixed maximal torus of $G$. Set $\varepsilon_{2i-1} \ (1 \le i \le n)$ the fundamental characters of $T$. Asuume $q \in \bb{C}^*$ is not a root of unity. We choose a semisimple element $s \in T$ such that $\varepsilon_{2i-1}(s) \in q^{\bb{Z}_{\rm{odd},\ge 0}}$ for any $i$ and put $d_{2i-1}=\{j|\varepsilon_{2j-1}(s)=q^{2i-1}\}$. Then the centralizer $G(s)$ of $s$ acts on
\[
\mf{g}_2:=\{X \in \mf{g} \ | \ sXs^{-1}=q^2X\}
\]
which has finitely many $G(s)$-oribits. Lusztig considered the category $\ms{Q}(\mf{g}_2)$ of semisimple $G(s)$-equivariant complex on $\mf{g}_2$ and constructed the canonical basis $\mb{B}(\mf{g}_2)$ of $K(\mf{g}_2)$ which is the Grothendieck group of $\ms{Q}(\mf{g}_2)$. \\
\quad On the other hand, let us consider the $\theta$-symmetric vector space $\mb{V}$ such that $\wt(\mb{V})=\sum_{i=1}^{n}d_{2i-1}(\alpha_{2i-1}+\alpha_{-2i+1})$ and the following $\theta$-quiver of type $A_{2n}$ and the $\theta$-orientation $\Omega$:
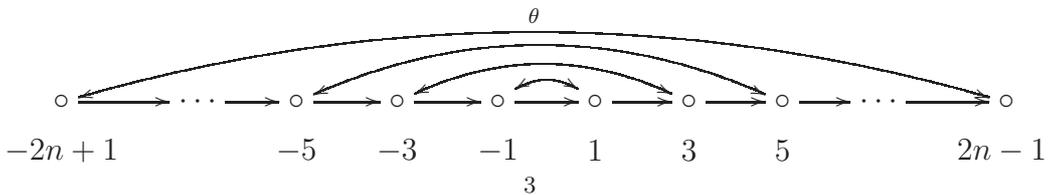
\[
\xymatrix@R=.8ex@C=3ex{
\circ\ar[r]\ar@/^2.2pc/@{<->}[rrrrrrrrr]^\theta&\ar[r]\cdots &\circ\ar[r]
\ar@/^1.8pc/@{<->}[rrrrr]&\circ\ar[r]\ar@/^1.2pc/@{<->}[rrr]&\circ\ar[r]
\ar@/^.7pc/@{<->}[r]&
\circ\ar[r]&\circ\ar[r]&\circ\ar[r]&\ar[r]\cdots&\circ \\
-2n+1& &-5&-3&-1&\;1\;&\;3\;&\;5\;&&2n-1\; 
}
\]
\quad In this case, we have $G(s)=\prod_{i=1}^{n}GL(d_{2i-1})=\tG{V}$ and $\mf{g}_2\cong\tEO{V}{\Omega}$. Thus the set ${}^\theta\!\ms{P}_{\mb{V},\Omega}$ of simple perverse sheaves conincide with $\mb{B}(\mf{g}_2)$.
\end{enumerate}

\end{rem}

\textbf{Acknowledgements.} \ I would like to thank Masaki Kashiwara, Georg Lusztig, Susumu Ariki, Syu Kato and Yuichiro Hoshi. Masaki Kashiwara guided me to Lusztig's geometric theory and gave me many advises, comments and patient encouragement. George Lusztig pointed out the relations between his work and this paper.  Syu Kato also commented on this paper and guided me to his geometric representation theory of AHA of type $C_n$ and certain quiver presentations. Susumu Ariki gave me some comments and encouragement. Yuichiro Hoshi taught me basic concepts, examples and some techniques in algebraic geometry and derived categories till midnights at RIMS.

\section{Preliminaries}
\subsection{Quantum enveloping algebras}\label{sec:qe}
\subsubsection{Quantum enveloping algebras and reduced $v$-analogue}
We shall recall the quantized universal enveloping algebra
$U_v(\mf{g})$. In this paper, we treat only the symmetric Cartan matrix case. 
Let $I$ be an index set (for simple roots), and $Q$ the free $\bb{Z}$-module with a basis $\{\alpha_i\}_{i\in I}$. Let $( \scbul,\scbul):Q\times Q\to\bb{Z}$ be
a symmetric bilinear form such that $(\al_i,\al_i)=2$ and $(\al_i,\al_j)\in\bb{Z}_{\le0}$ for $i\not=j$.
Let $v$ be an indeterminate and set 
$\mb{K}\seteq\bb{Q}(v)$.
We define its subrings $\mb{A}_0$, $\mb{A}_\infty$ and $\mb{A}$ as follows.
\begin{eqnarray*}
\mb{A}_0&=&\set{f\in\mb{K}}{\text{$f$ is regular at $v=0$}},\\
\mb{A}_\infty&=&\set{f\in\mb{K}}%
{\text{$f$ is regular at $v=\infty$}},\\
\mb{A}&=&\bb{Q}[v,v^{-1}].
\end{eqnarray*}
\begin{dfn}\label{U_v(g)}
The quantized universal enveloping algebra $U_v(\mf{g})$ is the $\mb{K}$-algebra
generated by elements $e_i,f_i$ and invertible
elements $t_i\ (i\in I)$
with the following defining relations.
\begin{enumerate}[{\rm(1)}]
\item The $t_i$'s commute with each other.
\item
$t_je_i\,t_j^{-1}=v^{(\al_j,\al_i)}\,e_i\ $ and $\ t_jf_it_j^{-1}=v^{-(\al_j,\al_i)}f_i\ $ for any $i,j\in I$.
\item\label{even} $\lbrack e_i,f_j\rbrack
=\delta_{ij}\dfrac{t_i-t_i^{-1}}{v-v^{-1}}$
for $i$, $j\in I$. 
\item {\rm(}{\em $v$-Serre relation}{\rm)} For $i\not= j$,
\begin{eqnarray*}
\sum^b_{k=0}(-1)^ke^{(k)}_ie_je^{(b-k)}_i=0, \ 
\sum^b_{k=0}(-1)^kf^{(k)}_i
f_jf_i^{(b-k)}=0.
\end{eqnarray*}
Here $b=1-(\al_i,\al_j)$ and
\begin{eqnarray*}
\ba{l}
e^{(k)}_i=e^k_i/\lbrack k\rbrack_v!\,,\; f^{(k)}_i=f^k_i/\lbrack k
\rbrack_v!\ , \ \lbrack k\rbrack_v=(v^k-v^{-k})/(v-v^{-1})\,, \ 
\lbrack k\rbrack_v!=\lbrack 1\rbrack_v\cdots \lbrack k\rbrack_v\,.
\ea
\end{eqnarray*}
\end{enumerate}
\end{dfn}
Let us denote by $\Uf$ the subalgebra of $\U$ generated by the $f_i$'s.

Let $e'_i$ and $e^*_i$ be the operators on $\Uf$ defined by
$$[e_i,a]=\dfrac{(e^*_ia)t_i-t_i^{-1}e'_ia}{v-v^{-1}}\quad(a\in\Uf).$$
These operators satisfy the following formulas similar to derivations:
\begin{eqnarray*}
e_i'(ab)=(e_i'a)b+(\Ad(t_i)a)e_i'b.
\end{eqnarray*}
The algebra $\Uf$ has a unique symmetric bilinear form $(\scbul,\scbul)$
such that $(1,1)=1$ and
\[
(e'_ia,b)=(a,f_ib)\quad\text{for any $a,b\in\Uf$.}
\]
It is non-degenerate.
The left multiplication operator $f_j$ and $e'_i$ 
satisfy the commutation relations
\[
e'_if_j=v^{-(\al_i,\al_j)}f_je'_i+\delta_{ij}, \ 
e_i^*f_j=f_je_i^*+\delta_{ij}\Ad(t_i),
\]
and the $e_i'$'s satisfy the $v$-Serre relations.

\begin{dfn}
The reduced $v$-analogue $B_v(\mf{g})$ of $\mf{g}$ is the $\bb{Q}(v)$-algebra generated by $e_i'$ and $f_i$. 
\end{dfn}

\subsubsection{Review on crystal bases and global bases of $U_v^-$}
Since $e_i'$ and $f_i$ satisfy the $v$-boson relation, any element $a \in U_v^{-}(\mf{g})$ can be uniquely  written as
\[
a=\smash{\sum_{n \ge 0}}f_i^{(n)}a_n \quad \text{with} \ e_i'a_n=0.
\]
Here $f_i^{(n)}=\dfrac{f_i^n}{[n]_v!}$. 
\begin{dfn}
We define the modified root operators $\widetilde{e}_i$ and $\widetilde{f}_i$ on $U_v^{-}(\mf{g})$ by 
\[
\widetilde{e}_ia=\sum_{n \ge 1}f_i^{(n-1)}a_n, \quad \widetilde{f}_ia=\sum_{n \ge 0}f_i^{(n+1)}a_n.
\] 
\end{dfn}
\begin{thm}[\cite{K}]
We define
\begin{eqnarray*}
L(\infty)&=&\sum_{\ell \ge 0,\;i_1, \ldots ,i_{\ell} \in I}
\mb{A}_0\tf_{i_1} \cdots \tf_{i_\ell} \cdot 1
 \subset U_v^{-}(\mf{g}), \\
B(\infty)&=&\set{\tf_{i_1} \cdots \tf_{i_\ell} \cdot 1\; \mod vL(\infty)}
{\ell \ge 0,i_1, \cdots ,i_{\ell} \in I} \subset L(\infty)/vL(\infty).
\end{eqnarray*}
Then we have
\bnum
\item $\widetilde{e}_iL(\infty) \subset L(\infty)$ 
and $\widetilde{f}_{i}L(\infty) \subset L(\infty)$, 
\item $B(\infty)$ is a basis of $L(\infty)/vL(\infty)$,
\item
$\widetilde{f}_iB(\infty) \subset B(\infty)$ 
and $\widetilde{e}_iB(\infty) \subset B(\infty) \cup \{0\}$. 
\enum
We call $(L(\infty),B(\infty))$ the crystal basis of $U_v^{-}(\mf{g})$.
\end{thm}

\begin{dfn}
We define $\varepsilon_i(b):=\max\{m \in \bb{Z}_{\ge 0}|\widetilde{e}_i^mb \neq 0\}$ for $i \in I$ and $b \in B(\infty)$.
\end{dfn}

Let $-$ be the automorphism of $\mb{K}$ sending $v$ to $v^{-1}$.
Then $\ol{\mb{A}_0}$ coincides with $\mb{A}_\infty$. 

Let $V$ be a vector space over $\mb{K}$,
$L_0$ an $\mb{A}$-submodule of $V$,
$L_\infty$ an $\mb{A}_\infty$- submodule, and
$V_\mb{A}$ an $\mb{A}$-submodule.
Set $E\seteq L_0\cap L_\infty\cap V_\mb{A}$.

\begin{dfn}[\cite{K}]
We say that $(L_0,L_\infty,V_\mb{A})$ is {\em balanced}
if each of $L_0$, $L_\infty$ and $V_\mb{A}$
generates $V$ as a $\mb{K}$-vector space,
and if one of the following equivalent conditions is satisfied.
\bnum
\item
$E \to L_0/v L_0$ is an isomorphism,
\item
$E \to L_\infty/v^{-1}L_\infty$ is an isomorphism,
\item
$(L_0\cap V_\mb{A})\oplus
(v^{-1} L_\infty \cap V_\mb{A}) \to V_\mb{A}$is an isomorphism.
\item
$\mb{A}_0\otimes_\bb{Q} E \to L_0$, $\mb{A}_\infty\otimes_\bb{Q} E \to L_\infty$,$\mb{A}\otimes_\bb{Q} E \to V_\mb{A}$ and $\mb{K} \otimes_\bb{Q} E \to V$ are isomorphisms.
\enum
\end{dfn}

Let $-$ be the ring automorphism of $\U$ sending
$v$, $t_i$, $e_i$, $f_i$ to $v^{-1}$, $t_i^{-1}$, $e_i$, $f_i$.

Let $\U_\mb{A}$ be the $\mb{A}$-subalgebra of
$\U$ generated by $e_i^{(n)}$, $f_i^{(n)}$
and $t_i$.
Similarly we define
$\Uf_\mb{A}$.

\begin{thm}
$(L(\infty),L(\infty)^-,\Uf_\mb{A})$ is balanced.
\end{thm}
Let 
\[
G^{\text{low}}\cl L(\infty)/v L(\infty)\isoto 
E\seteq L(\infty)\cap L(\infty)^-
\cap \Uf_\mb{A}
\] 
be the inverse of $E\isoto L(\infty)/v L(\infty)$.
Then $\set{G^{\text{low}}(b)}{b\in B(\infty)}$ forms a basis of $\Uf$.
We call it a (lower) {\em global basis}.
It is first introduced by G.\ Lusztig (\cite{Lus1})
under the name of ``canonical basis'' for the A, D, E cases.

\begin{dfn}
Let
\[
\set{G^\upper(b)}{b \in B(\infty)}
\]
be the dual basis of $\set{G^{\text{low}}(b)}{b \in B(\infty)}$ 
with respect to the inner product $( \scbul,\scbul)$. 
We call it the upper global basis of $U_v^{-}(\mf{g})$.
\end{dfn}

\subsection{Symmetric Crystals}
Let $\theta$ be an automorphism of
$I$ such that $\theta^2=\id$ and 
$(\al_{\theta(i)},\al_{\theta(j)})=(\al_i,\al_j)$.
Hence it extends to an automorphism of the root lattice $Q$
by $\theta(\al_i)=\al_{\theta(i)}$,
and induces an automorphism of $\U$.

\begin{dfn}\label{def:Bt}
Let $B_\theta(\mf{g})$ be the $\mb{K}$-algebra
generated by $E_i$, $F_i$, and
invertible elements $T_i (i\in I)$ satisfying the following defining relations:
\begin{enumerate}[{\rm(i)}]
\item the $T_i$'s commute with each other,
\item
$T_{\theta(i)}=T_i$ for any $i$,
\item
$T_iE_jT_i^{-1}=v^{(\al_i+\al_{\theta(i)},\al_j)}E_j$ and
$T_iF_jT_i^{-1}=v^{(\al_i+\al_{\theta(i)},-\al_j)}F_j$
for $i,j\in I$,
\item
$E_iF_j=v^{-(\al_i,\al_j)}F_jE_i+
(\delta_{i,j}+\delta_{\theta(i),j}T_i)$
for $i,j\in I$,
\item
the $E_i$'s and the $F_i$'s satisfy the $v$-Serre relations.
\end{enumerate}
\end{dfn}

We set $F_i^{(n)}=F_i^n/[n]_v!$.

\begin{prop}[{\cite[Proposition2.11.]{EK2}}]\label{prop:Vtheta}
Let 
\[
\lambda \in P_+\seteq\set{\lambda \in \Hom(Q,\bb{Q})}{
\text{$\lambda(\alpha_i) \in \bb{Z}_{\ge0}$ for any $i\in I$}}
\]
be a dominant integral weight such that
$\theta(\lambda)=\lambda$.
\begin{enumerate}[{\rm(i)}]
\item
There exists a $B_\theta(\mf{g})$-module $V_\theta(\lambda)$
generated by a non-zero vector $\vac_\lambda$ such that
\begin{enumerate}[{\rm(a)}]
\item $E_i\vac_\lambda=0$ for any $i\in I$,
\item $T_i\vac_\lambda=v^{(\al_i,\lambda)}\vac_\lambda$ for any $i\in I$,
\item $\set{u\in V_\theta(\lambda)}{\text{$E_iu=0$ for any $i\in I$}}
=\mb{K}\vac_\lambda$.
\end{enumerate}
Moreover such a $V_\theta(\lambda)$ is irreducible and
unique up to an isomorphism.
\item There exists a unique non-degenerate symmetric bilinear form $(\scbul,\scbul)$
on $V_\theta(\lambda)$ such that $(\vac_\lambda,\vac_\lambda)=1$ and
$(E_iu,v)=(u,F_iv)$ for any $i\in I$ and $u,v\in V_\theta(\lambda)$.
\item There exists an endomorphism $-$ of $V_\theta(\lambda)$
such that $\ol{\vac_\lambda}=\vac_\lambda$ and
$\ol{av}=\bar{a}\bar{v}$, $\ol{F_iv}=F_i\bar{v}$ for any $a\in \mb{K}$
and $v\in V_\theta(\lambda)$.
\end{enumerate}
\end{prop}

Hereafter we assume further that
\[
\text{\em there is no $i\in I$ such that $\theta(i)=i$.}
\]
In \cite{EK2}, we conjectured that $V_\theta(\la)$ has a crystal basis.
This means the following.
Since $E_i$ and $F_i$ satisfy the $v$-boson relation $E_iF_i=v^{-(\al_i,\al_i)}F_iE_i+1$, 
we define the modified root operators:
\[
\tEi(u)=\sum_{n\ge1}F_i^{(n-1)}u_n\ 
\text{and}\ 
\tFi(u)=\sum_{n\ge0}F_i^{(n+1)}u_n,
\]
when writing $u=\sum_{n\ge0}F_i^{(n)}u_n$ with $E_iu_n=0$.
Let $L_\theta(\la)$ be the $\mb{A}_0$-submodule of $V_\theta(\la)$ generated by $\widetilde{F}_{i_1} \cdots \widetilde{F}_{i_\ell}\vac_\lambda$ ($\ell\ge 0$ and $i_1,\ldots,i_\ell\in I$),
and let $B_\theta(\la)$ be the subset 
\[
\set{\widetilde{F}_{i_1}\cdots\widetilde{F}_{i_\ell}\vac_\lambda \mod vL_\theta(\la)}{\ell \ge 0,i_1,\ldots, i_\ell \in I}
\]
of $L_\theta(\la)/v L_\theta(\la)$.
\begin{conj}\label{conj:crystal}
Let $\la$ be a dominant integral weight such that $\theta(\la)=\la$.
\begin{enumerate}
\item
$\tFi L_\theta(\la)\subset L_\theta(\la)$
and $\tEi L_\theta(\la)\subset L_\theta(\la)$,
\item
$B_\theta(\la)$ is a basis of $L_\theta(\la)/vL_\theta(\la)$,
\item
$\tFi B_\theta(\la)\subset B_\theta(\la)$,
and
$\tEi B_\theta(\la)\subset B_\theta(\la)\sqcup\{0\}$,
\item
$\tFi\tEi(b)=b$ for any $b\in  B_\theta(\la)$ such that $\tEi b\not=0$,
and $\tEi\tFi(b)=b$ for any $b\in  B_\theta(\la)$.
\end{enumerate}
\end{conj}
Moreover we conjectured that
$V_\theta(\la)$ has a global crystal basis.
Namely we have
\begin{conj}\label{conj:bal}
$(L_\theta(\la),\overline{L_\theta(\la)},V_\theta(\la)^\text{low}_\mb{A})$
is balanced.
Here $V_\theta(\la)^\text{low}_\mb{A}\seteq\Uf_{\mb{A}}\vac_\lambda$.
\end{conj}

\begin{exa}\label{exa:A}
Suppose $\mf{g}=\mf{gl}_{\infty}$, the Dynkin diagram involution $\theta$ of $I$ defined by $\theta(i)=-i$ for $i \in I=\bb{Z}_{\text{odd}}$.
{\scriptsize
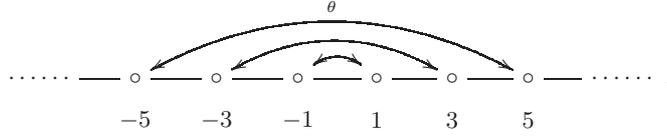
$$
\xymatrix@R=.8ex@C=3ex{
\cdots\cdots\ar@{-}[r]&\circ\ar@{-}[r]
\ar@/^1.8pc/@{<->}[rrrrr]^\theta&\circ\ar@{-}[r]\ar@/^1.2pc/@{<->}[rrr]&\circ\ar@{-}[r]
\ar@/^.7pc/@{<->}[r]&
\circ\ar@{-}[r]&\circ\ar@{-}[r]&\circ\ar@{-}[r]&\cdots\cdots\ .\\
&-5&-3&-1&\;1\;&\;3\;&\;5\;
}
$$} 
And assume $\lambda=0$. In this case, we can prove 
\[
V_\theta(0) \cong U_v^-/\sum_{i \in I}U_v^-(f_i-f_{\theta(i)}).
\]
Moreover we can construct a PBW type basis, a crystal basis and an upper and lower global basis on $V_\theta(0)$ parametrized by "the $\theta$-restricted multisegments". For more details, see \cite{EK2}.
\end{exa}
\subsection{Criterion for crystals}\label{sec-cry}

Let $\mb{K}[e,f]$ be the ring generated by $e$ and $f$ with the defining relation $ef=v^{-2}fe+1$. We call this algebra the $v$-boson algebra. Let $P$ be a free $\bb{Z}$-module, and let $\alpha$ be a non-zero element of $P$. Let $M$ be a $\mb{K}[e,f]$-module. Assume that $M$ has a weight decomposition $M=\oplus_{\xi \in P}M_\xi$ and $eM_\lambda \subset M_{\lambda+\alpha}$ and $fM_\lambda \subset M_{\lambda-\alpha}$. Asuume the following finiteness conditions:
\[
\text{for any $\lambda \in P$, $\dim{M_\lambda}<\infty$ and $M_{\lambda+n\alpha}=0$ for $n \gg 0$}.
\]
Hence for $u \in M$, we can write $u=\sum_{n \ge 0}f^{(n)}u_n$ with $eu_n=0$. We define endmorphisms $\widetilde{e}$ and $\widetilde{f}$ of $M$ by
\[
\widetilde{e}u=\sum_{n \ge 1}f^{(n-1)}u_n, \quad \widetilde{f}u=\sum_{n \ge 0}f^{(n+1)}u_n.
\]
Let $B$ be a crystal with weight decomposition by $P$ in the following sense.
We have $\wt\colon B \to P$, $\widetilde{f}\colon B \to B$, $\widetilde{e}\colon B \to B \sqcup \{0\}$ and $\varepsilon\colon B \to \bb{Z}_{\ge 0}$ satisfying the following properties, where $B_\lambda=\wt^{-1}(\lambda)$:  
\begin{enumerate}[{\rm(i)}]
\item $\widetilde{f}B_\lambda \subset B_{\lambda-\alpha}$ and $\widetilde{e}B_\lambda \subset B_{\lambda+\alpha} \sqcup \{0\}$ for any $\lambda \in P$,
\item $\widetilde{f}\widetilde{e}b=b$ if $\widetilde{e}b \neq 0$, and $\widetilde{e} \circ \widetilde{f}=\id_B$,
\item for any $\lambda \in P$, $B_\lambda$ is a finite set and $B_{\lambda+n\alpha}=\phi$ for $n \gg 0$,
\item $\varepsilon(b)=\max\{n \ge 0 \ | \ \widetilde{e}^nb \neq 0\}$ for any $b \in B$.
\end{enumerate}
Set $\ord(a)=\sup\{n \in \bb{Z} \ | \ a \in v^n\mb{A}_0\}$ for $a \in \mb{K}$. We understand $\ord(0)=\infty$. \\
\quad Let $\{G(b)\}_{b \in B}$ be a system of generators of $M$ with $G(b) \in M_{\wt(b)}$. Asuume that we have expressions: 
\[
eG(b)=\sum_{b' \in B}E_{b,b'}G(b), \quad 
fG(b)=\sum_{b' \in B}F_{b,b'}G(b).
\]
Now consider the following conditions for these data, where $\ell=\varepsilon(b)$ and $\ell'=\varepsilon(b')$:
\begin{eqnarray}
&{}&\ord(F_{b,b'}) \ge 1-\ell', \label{c1}\\
&{}&\ord(E_{b,b'}) \ge -\ell', \\
&{}&F_{b,\widetilde{f}b} \in v^{-\ell}(1+v\bb{A}_0), \\
&{}&E_{b,\widetilde{f}b} \in v^{1-\ell}(1+v\bb{A}_0), \\
&{}&\text{$\ord(F_{b,b'})>1-\ell'$ if $\ell<\ell'$ and $b' \neq \widetilde{f}b$}, \\
&{}&\text{$\ord(E_{b,b'})>-\ell'$ if $\ell<\ell'+1$ and $b' \neq \widetilde{e}b$}. \label{c6}
\end{eqnarray}
\begin{thm}[{\cite[Theorem 4.1, Corollary 4.4]{EK2}}]\label{cricry}
Assume the conditions (\ref{c1})--(\ref{c6}). Let $L$ be the $\mb{A}_0$-submodule $\sum_{b \in B}\mb{A}_0G(b)$ of $M$. Then we have $\widetilde{e}L \subset L$ and $\widetilde{f}L \subset L$. Moreover we have
\[
\widetilde{e}G(b) \equiv G(\widetilde{e}b) \mod{vL}, \quad 
\widetilde{f}G(b) \equiv G(\widetilde{f}b) \mod{vL}
\]
for any $b \in B$. Here we understand $G(0)=0$. \\
\end{thm}
In \cite{EK2}, this theorem is proved under more general assumptions.

\subsection{Perverse Sheaves}
\subsubsection{Perverse Sheaves}
In this paper, we consider algebraic varieties over $\bb{C}$. Let $\ms{D}(X)$ be the bounded derived category of constructible complexes of sheaves on an algebraic variety $X$. We denote by $\ms{D}^{\le{0}}(X)$ (resp. $\ms{D}^{\ge{0}}(X)$) the full subcategory of $\ms{D}(X)$ consisting of objects $L$ satisfying $H^k(L)=0$ for $k>0$ (resp. $k<0$). Put $\ms{D}^{\le{n}}=\ms{D}^{\le{0}}[-n]$ and $\ms{D}^{\ge{n}}=\ms{D}^{\ge{0}}[-n]$. \\
\quad For a morphism $f\colon X \to Y$ of algebraic varieties $X$ and $Y$, let $f^*$ be the inverse image, $f_!$ the direct image with proper support and $D\colon \ms{D}(X) \to \ms{D}(X)$ the Verdier duality functor. 
\begin{lem}\label{lem-VD} \ 
\begin{enumerate}[{(i)}]
\item Suppose that $f\colon X \to Y$ is smooth with the fiber dimension $d$. Then $D(f^*L) \cong f^*D(L)[2d]$ for $L \in \ms{D}(Y)$.
\item Suppose that $f\colon X \to Y$ is proper. Then $D(f_!L) \cong f_!D(L)$ for $L \in \ms{D}(X)$.
\end{enumerate}
\end{lem}
Let $({}^p\ms{D}^{\le 0}(X),{}^p\ms{D}^{\ge 0}(X))$ be the perverse $t$-structure and $\Perv(X)\seteq{}^p\ms{D}^{\le 0}(X) \cap {}^p\ms{D}^{\ge 0}(X)$.
\begin{lem}\label{lem:hom}
Suppose $L \in {}^p\ms{D}^{\le 0}(X)$ and $K \in {}^p\ms{D}^{\ge 0}(X)$, then $H^j(\rHom(L,K))=0$ for $j<0$, namely $\rHom(L,K) \in \ms{D}^{\ge 0}(X)$.
\end{lem}
\quad Let ${}^p\!H^k( \ )$ be the $k$-th perverse cohomology sheaf. We say that an object $L$ in $\ms{D}(X)$ is semisimple if $L$ is isomorphic to the direct sum $\oplus_{k}{}^p\!H^k(L)[-k]$ and if each ${}^p\!H^k(L)$ is a semisimple perverse sheaf. Assume that we are given an action of a connected algebraic group $G$ on $X$. A semisimple object $L$ in $\ms{D}(X)$ is said to be $G$-equivariant if each ${}^p\!H^i(L)$ is a $G$-equivariant perverse sheaf. 
\begin{lem}\label{lem-pv} \ 
\begin{enumerate}[{(i)}]
\item Suppose that $f\colon X \to Y$ is smooth with connected fibers of dimension $d$. Then we have a fully faithful functor $\Perv(Y) \to \Perv(X)$ given by $K \mapsto f^*K[d]$. Moreover if $K$ is simple, then $f^*K[d]$ is simple.
\item Let $G$ be a connected algebraic group of dimension $d$ and $\Perv_G(X)$ the category of $G$-equivariant perverse sheaves. Suppose that $f\colon X \to Y$ is a principal $G$-bundle. The functors
\[
\Perv(Y) \to \Perv_G(X)\colon K \mapsto f^*K[d]
\]
and 
\[
\Perv_G(X) \to \Perv(Y)\colon L \mapsto ({}^pH^{-d}f_*L)
\]
define an equivalence of categories, quasi-inverse to each other. \\
\quad Moreover if $K$ is a semisimple object of $\ms{D}(Y)$, then $f^*K$ is a $G$-equivariant semisimple object in $\ms{D}(X)$. Conversely, if $L$ is a $G$-equivariant semisimple object of $\ms{D}(X)$, then there is a unique semisimple object $K \in \ms{D}(Y)$ such that $L \cong f^*K$. 
\end{enumerate}
\end{lem}
\quad We denote by $\mb{1}_X$ the constant sheaf on $X$. 
\begin{lem}[\cite{BBD}, \cite{Lus3}]\label{lem:perv} \ 
\begin{enumerate}[{(\rm 1)}]
\item Let $f\colon X \to Y$ be a projective morphism with $X$ smooth. Then $f_!\mb{1}_X \in \ms{D}(Y)$ is semisimple.
\item Let $f\colon  X \to Y$ be a morphism. Assume that there exists a partition $X=X_0 \sqcup X_1 \sqcup \cdots \sqcup X_m$ such that $X_{\le j}=X_0 \cup X_1 \cup \cdots \cup X_j$ is closed for $j=0,1, \ldots ,m$. Assume that, for each $j$, the restriction $f_j\colon X_j \to Y$ of $f$ decomposes as $X_j \stackrel{f''_j}{\longrightarrow} Z_j \stackrel{f'_j}{\longrightarrow} Y$ such that $Z_j$ is smooth, $f''_j$ is an affine bundle and $f'_j$ is projective. Then $f_!\mb{1}_X \in \ms{D}(Y)$ is semisimple. Moreover, we have $f_!\mb{1}_X\cong\bigoplus_{j}(f_j)_!\mb{1}_{X_j}$.
\end{enumerate}
\end{lem}
\subsubsection{Simple objects}
Let $Y$ be an irreducible variety and $U$ a Zariski open subset of $Y$. Set $Z:=Y \backslash U$ and $i\colon Z \hookrightarrow Y$. 
\begin{prop}
For $F \in \rm{Perv}(U)$, there exists a unique perverse sheaf ${}^\pi\!F$ on $Y$ satisfying 
\begin{enumerate}[(i)]
\item ${}^\pi\!F|_U \cong F$, 
\item $i^*({}^\pi\!F) \in {}^p\!\ms{D}^{\le -1}(Z)$,
\item $i^!({}^\pi\!F) \in {}^p\!\ms{D}^{\ge 1}(Z)$.
\end{enumerate}
We call ${}^\pi\!F$ the minimal extension of $F$. We have the following properties of the minimal extension:
\begin{enumerate}[(1)]
\item ${}^\pi\!F$ has neither non-trivial subobject nor non-trivial quotient object whose support is contained in $Z$.
\item If $F$ is simple, then ${}^\pi\!F$ is simple.
\item For the Verdier duality functors $D_Y$ and $D_U$, we have $D_Y({}^\pi\!F) \cong {}^\pi\!(D_U(F))$.
\end{enumerate}
\end{prop}
Let $X$ be a variety, $Y$ an irreducible locally closed smooth subvariety of $X$. For a simple local system $L$ on $Y$, the minimal extension ${}^\pi\!L[\dim{Y}]$ is called the intersection cohomology complex of $Y$. We can regard ${}^\pi\!L[\dim{Y}]$ as a simple perverse on $X$ whose support is the closure $\overline{Y}$ of $Y$. Conversely, any simple object in $\rm{Perv}(X)$ is obtained in this way.
\begin{thm}[\cite{BBD}]
For a simple perverse sheaf $F$ on $X$, there exist an irreducible closed subvariety $Y$ and an simple local system $L$ on $Y$ such that $F \cong {}^\pi\!L[\dim{Y}]$. Moreover, for simple perverse sheaves $F_1$ and $F_2$, we have $\rm{Ext}^0(F_1,F_2)=\Hom_{\rm{Perv}(X)}(F_1,F_2)=\bb{C}$ or $0$ according that $F_1$ and $F_2$ are isomorphic or not.

\end{thm}
\subsubsection{Fourier-Sato-Deligne transforms}
Let $E \to S$ be a vector bundle and $E^* \to S$ the dual vector bundle. Hence $\bb{C}^\times$ acts on $E$ and $E^*$. We say that $L \in \ms{D}(E)$ is monodromic if $H^j(L)$ is locally constant on  every $\bb{C}^*$-orbit of $E$. Let $\ms{D}_{\text{mono}}(E)$ be the full subcategory of $\ms{D}(E)$ consisting of monodromic objects. Then we can define the Fourier transform 
\[
\Phi_{E/S}\colon \ms{D}_{\text{mono}}(E) \to \ms{D}_{\text{mono}}(E^*).
\]
We will use the following properties of $\Phi$.
\begin{prop}[e.g. \cite{KS}, \cite{Lau}]\label{FD}
\quad
\begin{enumerate}[{\rm(1)}]
\item For $K \in \ms{D}_{\mathrm{mono}}(E)$, we have $\Phi_{E^*/S} \circ \Phi_{E/S}(K) \cong a^*K$, where $a\colon E \to E$ is the multiplication by $-1$ on each fiber of $E$.
\item For a perverse sheaf $K \in \ms{D}_{\mathrm{mono}}(E)$, $\Phi_{E/S}(K)$ is a perverse sheaf in $\ms{D}_{\mathrm{mono}}(E^*)$.
\item Let $E_1$ and $E_2$ be two vector bundles over $S$ with rank $r_1$ and $r_2$. Let $f\colon E_1 \to E_2$ be a morphism of vector bundles and ${}^tf\colon E_2^* \to E_1^*$ the transpose of $f$. Then we have 
\[
\Phi_{E_2/S} \circ f_! \cong ({}^tf)^* \circ \Phi_{E_1/S}[r_2-r_1], \quad 
({}^tf)_! \circ \Phi_{E_2/S} \cong \Phi_{E_1/S} \circ f^*[r_1-r_2].
\]
\item Suppose that $E_1 \to S_1$ and $E \to S$ are two vector bundles. If the following two diagrams
\[
\xymatrix{
E_1\ar[d]\ar[r]^{f_E} & E\ar[d] \\
S_1\ar[r]^{\rho} & S
} \quad 
\xymatrix{
E_1^*\ar[d]\ar[r]^{f_{E^*}} & E^*\ar[d] \\
S_1\ar[r]^{\rho} & S
}
\]
are Cartesian, then we have
\[
\Phi_{E/S} \circ (f_E)_!\cong(f_{E^*})_! \circ \Phi_{E_1/S_1}, \quad 
\Phi_{E_1/S_1} \circ (f_E)^*\cong(f_{E^*})^* \circ \Phi_{E/S}.
\]
\item The Fourier transforms commute with the Verdier duality functors.
\end{enumerate}
\end{prop}

\subsection{Quivers}
Let $I$ and $\alpha_i$'s be as in \ref{sec:qe}.
\begin{dfn}
A quiver $(I,H)$ associated with the symmetric Cartan matrix is a following data:
\bnum[{\rm(i)}]
\item a set $H$,
\item two maps $\vo, \vi\colon H \to I$ such that $\vo(h) \neq \vi(h)$ for any $h \in H$,
\item an involution $h \mapsto \bar{h}$ on $H$ satisfying $\vo(\bar{h})=\vi(h)$ and $\vi(\bar{h})=\vo{h}$,
\item $\sharp\{h \in H|\vo(h)=i,\vi(h)=j\}=-(\alpha_i,\alpha_j)$ for $i \neq j$.
\enum
\quad An orientation of a quiver $(I,H)$ is a subset $\Omega$ of $H$ such that $\Omega \cap \ol{\Omega}=\phi$ and $\Omega \cup \ol{\Omega}=H$. \\
\quad For a fixed orientation $\Omega$, we call a vertex $i \in I$ a sink if $\vo(h) \neq i$ for any $h \in \Omega$.
\end{dfn}

\begin{dfn}
Let $\cal{V}$ be the category of $I$-graded vector spaces $\mb{V}=(\mb{V}_i)_i$ with morphisms being linear maps respecting the grading. Put $\wt(\mb{V})=\sum_{i \in I}(\dim{\mb{V}}_i)\alpha_i$.
\end{dfn}
Let $\mb{S}_i$ be an $I$-graded vector space such that $\wt(\mb{S}_i)=\alpha_i$ .
\begin{dfn}
For $\mb{V} \in \cal{V}$ and a subset $\Omega$ of $H$, we define 
\[
\EO{V}{\Omega}\colon =\bigoplus_{h \in \Omega}\Hom(\mb{V}_{\vo(h)},\mb{V}_{\vi(h)}).
\]
\quad The algebraic group $\mb{G}_{\mb{V}}=\prod_{i \in I}GL(\mb{V}_i)$ acts on $\EO{V}{\Omega}$ by $(g,x) \mapsto gx$ where $(gx)_h=g_{\vi(h)} x_h g_{\vo(h)}^{-1}$.\\
\quad The group $(\bb{C}^\times)^{\Omega}$ also acts on $\EO{V}{\Omega}$ by $x_h \mapsto c_hx_h \ (h \in \Omega, c_h \in \bb{C}^\times)$. \\
\quad For $x \in \EO{V}{\Omega}$, an $I$-graded subspace $\mb{W} \subset \mb{V}$ is $x$-stable if $x_h(\mb{W}_{\vo(h)}) \subset \mb{W}_{\vi(h)}$ for any $h \in \Omega$.  
\end{dfn}
Note that $E_{\mb{S}_i,\Omega} \cong \{\pt\}$. \\

\section{A Review on Lusztig's Geometric Construction}\label{sec:Lus}
We give a quick review on Lusztig's theory in \cite{Lus1} and \cite{Lus2} (cf. \cite{Lus3}). For a sequence $\mb{i}=(i_1, \ldots ,i_m) \in I^m$ and a sequence $\mb{a}=(a_1, \ldots ,a_m) \in \bb{Z}_{\ge 0}^m$, a flag of type $(\mb{i},\mb{a})$ is by definition a finite decreasing sequence $F=(\mb{V}=\mb{F}^0 \supset \mb{F}^1 \supset \cdots \supset \mb{F}^m=\{0\})$ of $I$-graded subspaces of $\mb{V}$ such that the $I$-graded vector space $\mb{F}^{\ell-1}/\mb{F}^\ell$ vanishes in degrees $\neq{i_\ell}$ and has dimension $a_\ell$ in degree $i_\ell$. We denote by $\tFl{i}{a}{\Omega}$ the set of pairs $(x,F)$ such that $x \in \EO{V}{\Omega}$ and $F$ is an $x$-stable flag of type $(\mb{i},\mb{a})$. The group $G_{\mb{V}}$ acts on $\tFl{i}{a}{\Omega}$. The first projection $\pi_{\mb{i},\mb{a}}\colon \tFl{i}{a}{\Omega} \to \EO{V}{\Omega}$ is a $G_{\mb{V}}$-equivariant projective morphism. \\
\quad By Lemma \ref{lem:perv}, $\LL{i}{a}{\Omega}\colon =(\pi_{\mb{i},\mb{a}})_!(\mb{1}_{\tFl{i}{a}{\Omega}}) \in \ms{D}(\EO{V}{\Omega})$ is a semisimple complex. We define $\mP{V}{\Omega}$ as the set of the isomorphism classes of simple perverse sheaves $L \in \ms{D}(\EO{V}{\Omega})$ satisfying the following property: $L$ appears as a direct summand of $\LL{i}{a}{\Omega}[d]$ for some $d$ and $(\mb{i},\mb{a})$. We denote by $\mQ{V}{\Omega}$ the full subcategory of $\ms{D}(\EO{V}{\Omega})$ consisting of all objects which are isomorphic to finite direct sums of complexes of the form $L[d]$ for various $L \in \mP{V}{\Omega}$ and various integers $d$. Any complex in $\mP{V}{\Omega}$ is $G_{\mb{V}} \times (\bb{C}^\times)^{\Omega}$-equivariant. \\
\quad Let $\mb{T},\mb{W},\mb{V}$ be $I$-graded vector spaces such that $\wt(\mb{V})=\wt(\mb{W})+\wt(\mb{T})$. We consider the following diagram
\[
\xymatrix{
\EO{T}{\Omega} \times \EO{W}{\Omega} & \pE{\Omega}\ar[l]_(.3){p_1}\ar[r]^{p_2} & \ppE{\Omega}\ar[r]^{p_3} & \EO{V}{\Omega}.
}
\]
Here $\ppE{\Omega}$ is the variety of $(x,W)$ where $x \in \E{\mb{V},\Omega}$ and $W$ is an $x$-stable $I$-graded subspace of $\mb{V}$ such that $\ud{W}=\ud{\mb{W}}$. The variety $\pE{\Omega}$ consists of $(x,W,\varphi^\mb{W},\varphi^{\mb{T}})$ where $(x,W) \in \ppE{\Omega}$, $\varphi^{\mb{W}}\colon \mb{W} \cong W$, and $\varphi^{\mb{T}}\colon \mb{T} \cong \mb{V}/W$. The morphisms $p_1,p_2$ and $p_3$ are given by $p_1(x,W,\varphi^{\mb{W}},\varphi^{\mb{T}})=(x|_\mb{T},x|_\mb{W})$, $p_2(x,W,\varphi^{\mb{W}},\varphi^{\mb{T}})=(x,W)$ and $p_3(x,W)=x$. Then $p_1$ is smooth with connected fibers, $p_2$ is a principal $G_{\mb{T}} \times G_{\mb{W}}$-bundle, and $p_3$ is projective. For a $G_{\mb{T}}$-equivariant semisimple complex $K_\mb{T}$ and a $G_{\mb{W}}$-equivariant semisimple complex $K_\mb{W}$, there exists a unique semisimple complex $K''$ satisfying $p_1^*(K_\mb{T} \boxtimes K_\mb{W})=p_2^*K''$. We define $K_\mb{T}*K_{\mb{W}}\colon =(p_3)_!(K'') \in\ms{D}(\EO{V}{\Omega})$. \\
\quad For an $I$-graded subspace $\mb{U}$ of $\mb{V}$ such that $\mb{V}/\mb{U} \cong \mb{T}$, we also consider the following diagram
\[
\xymatrix{
\EO{T}{\Omega} \times \EO{U}{\Omega} & \EE{U,V}{\Omega}\ar[l]_(.45){p}\ar@{^{(}->}[r]^(.6){\iota} & \EO{V}{\Omega}.
}
\]
Here $\EE{U,V}{\Omega}$ is the variety of $x \in \EO{V}{\Omega}$ such that $\mb{U}$ is $x$-stable. For $K \in \ms{D}(\EO{V}{\Omega})$, we define $\Res_{\mb{T},\mb{U}}(K)\colon =p_!\iota^*(K)$. \\
\quad We define $K_{\mb{V},\Omega}$ as the Grothendieck group of $\mQ{V}{\Omega}$. It is the additive group generated by the isomorphism classes $(L)$ of objects $L \in \mQ{V}{\Omega}$ with the relation $(L)=(L')+(L'')$ when $L \cong L' \oplus L''$. The group $K_{\mb{V},\Omega}$ has a $\bb{Z}[v,v^{-1}]$-module structure by $v(L)=(L[1])$ and $v^{-1}(L)=(L[-1])$ for $L \in \mQ{V}{\Omega}$. Hence, $K_{\mb{V},\Omega}$ is a free $\bb{Z}[v,v^{-1}]$-module with a basis $\{(L)|L \in \mP{V}{\Omega}\}$. We define $K_{\Omega}\colon =\bigoplus_{\mb{V}}K_{\mb{V},\Omega}$ where $\mb{V}$ runs over the isomorphism classes of $I$-graded vector spaces. Recall that $\mb{S}_i$ is an $I$-graded vector space such that $\wt(\mb{S}_i)=\alpha_i$. Then we can define the induction $f_i\colon K_{\mb{W},\Omega} \to K_{\mb{V},\Omega}$ and the restriction $e'_i\colon K_{\mb{V},\Omega} \to K_{\mb{W},\Omega}$ by 
\[
f_i(K)\colon =v^{\dim{\mb{W}_i}+\sum_{i \stackrel{\Omega}{\longrightarrow} j}\dim{\mb{W}_j}}(\mb{1}_{\mb{S}_i}*K), \quad  
e'_i(K)\colon =v^{-\dim{\mb{W}_i}+\sum_{i \stackrel{\Omega}{\longrightarrow} j}\dim{\mb{W}_j}}\Res_{\mb{S}_i,\mb{V}}(K).
\]
Then Lusztig's main theorem is stated as follows.
\begin{thm}[Lusztig] \quad
\bnum[(i)]
\item The operators $e_i'$ and $f_i$ define the action of the reduced $v$-analogue $B_v(\mf{g})$ of $\mf{g}$ on $K_{\Omega} \otimes_{\bb{Z}[v,v^{-1}]} \bb{Q}(v)$. The $B_v(\mf{G})$-module $K_{\Omega} \otimes_{\bb{Z}[v,v^{-1}]} \bb{Q}(v)$ is isomorphic to $U_v^-(\mf{g})$. The involution induced by the Verdier duality functor coincides with the bar involution on $U_v^-(\mf{g})$.
\item The simple perverse sheaves in $\sqcup_{\mb{V}}\mP{V}{\Omega}$ give a lower global basis of $U_v^-(\mf{g})$.
\enum
\end{thm}
\section{Quivers with an Involution $\theta$}
\subsection{Quivers with an involution $\theta$}
\begin{dfn}\label{def-tq}
A $\theta$-quiver is a data: 
\bnum
\item a quiver $(I,H)$,
\item involutions $\theta\colon I \to I$ and $\theta\colon H \to H$,
\enum
satisfying 
\bnum[{(a)}]
\item $\vo(\theta(h))=\theta(\vi(h))$ and $\vi(\theta(h))=\theta(\vo(h))$,
\item If $\theta(\vo(h))=\vi(h)$, then $\theta(h)=h$,
\item $\theta(\overline{h})=\overline{\theta(h)}$,
\item There is no $i\in I$ such that $\theta(i)=i$
\enum
A $\theta$-orientation is an orientation of $(I,H)$ such that $\Omega$ is stable by $\theta$.
\end{dfn}
From the assumption (d), any vertex $i$ is a sink with respect to some $\theta$-orientation $\Omega$.
\begin{exa}
We give two $\theta$-orientations for the case of Example \ref{exa:A}. The vertex $1$ is a sink in the right example.
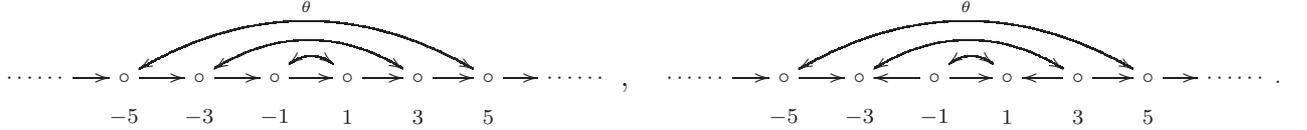
\begin{eqnarray*}
{\tiny
\xymatrix@R=.8ex@C=3ex{
\cdots\cdots\ar[r]&\circ\ar[r]
\ar@/^1.8pc/@{<->}[rrrrr]^\theta&\circ\ar[r]\ar@/^1.2pc/@{<->}[rrr]&\circ\ar[r]
\ar@/^.7pc/@{<->}[r]&
\circ\ar[r]&\circ\ar[r]&\circ\ar[r]&\cdots\cdots\ \\
&-5&-3&-1&\;1\;&\;3\;&\;5\;
}, \quad 
\xymatrix@R=.8ex@C=3ex{
\cdots\cdots\ar[r]&\circ\ar[r]
\ar@/^1.8pc/@{<->}[rrrrr]^\theta&\circ\ar@/^1.2pc/@{<->}[rrr]&\circ\ar[l]\ar[r]
\ar@/^.7pc/@{<->}[r]&
\circ&\circ\ar[r]\ar[l]&\circ\ar[r]&\cdots\cdots\ .\\
&-5&-3&-1&\;1\;&\;3\;&\;5\;
}
}
\end{eqnarray*}
\end{exa}

\begin{exa}
Our definition of a $\theta$-quiver contains the case of type $A_1^{(1)}$. The following three figures are three $\theta$-orientations in this case.
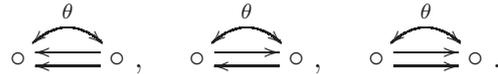
\[
\xymatrix{
\circ\ar@/^1pc/@{<->}[r]^\theta & \circ\ar@<-0.5ex>[l]\ar@<0.5ex>[l]
},\quad 
\xymatrix{
\circ\ar@<0.5ex>[r]\ar@/^1pc/@{<->}[r]^\theta & \circ\ar@<0.5ex>[l]
},\quad 
\xymatrix{
\circ\ar@<-0.5ex>[r]\ar@<0.5ex>[r]\ar@/^1pc/@{<->}[r]^\theta & \circ
}.
\]
\end{exa}

\begin{dfn}
\quad A $\theta$-symmetric $I$-graded vector space $\mb{V}$ is an $I$-graded vector space endowed with a non-degenerate symmetric bilinear form $(\scbul , \scbul)\colon \mb{V} \times \mb{V} \to \bb{C}$ such that $\mb{V}_i$ and $\mb{V}_j$ are orthogonal if $j \neq \theta(i)$. For an $I$-graded subspace $\mb{W}$ of $\mb{V}$, we set
\[
\mb{W}^{\bot}\colon =\{v \in \mb{V} \ | \ \text{$(v,w)=0$ for any $w \in \mb{W}$}\}.
\]
Hence $(\mb{W}^\bot)_{\theta(i)} \cong (\mb{V}_i/\mb{W}_i)^*$. 
\end{dfn}
Note that if $\mb{W} \supset \mb{W}^{\bot}$, then $\mb{W}/\mb{W}^\bot$ has a structure of $\theta$-symmetric $I$-graded vector space. Note that two $\theta$-symmetric $I$-graded vector spaces with the same dimension are isomorphic. 
\begin{dfn}\label{def-te}
Let $(I,H)$ be a $\theta$-quiver. For a $\theta$-symmetric $I$-graded vector space $\mb{V}$ and a $\theta$-stable subset $\Omega$ of $H$, we define
\[
\tEO{V}{\Omega}\colon =\{x \in \EO{V}{\Omega} \ | \ \text{$x_{\theta(h)}=-{}^tx_{h} \in \Hom(\mb{V}_{\theta(\vi(h))}, \mb{V}_{\theta(\vo(h))})$ for any $h \in \Omega$}\}.
\]
The algebraic group $\tG{V}\colon =\{g \in \mb{G}_\mb{V} \ | \ \text{${}^tg_i^{-1}=g_{\theta(i)}$ for any $i$}\}$ naturally acts on $\tEO{V}{\Omega}$. \\
\quad Set $(\bb{C}^\times)^{\Omega,\theta}\colon =\{(c_h)_{h \in \Omega} \ | \ \text{$c_h \in \bb{C}^\times$ and $c_{\theta(h)}=c_h$}\}$.  The group $(\bb{C}^\times)^{\Omega,\theta}$ also acts on $\tEO{V}{\Omega}$ by $x_h \mapsto c_hx_h \ (h \in \Omega)$. These two actions commute with each other. 
\end{dfn}
\begin{dfn}
For a $\theta$-symmetric $I$-graded vector space $\mb{V}$, a sequence $\mb{i}=(i_1, \ldots ,i_{2m}) \in I^{2m}$ such that $\theta(i_\ell)=i_{2m-\ell+1}$ and a sequence $\mb{a}=(a_1, \ldots ,a_{2m}) \in \bb{Z}_{\ge 0}^{m}$ such that $a_{2m-\ell+1}=a_\ell$, we say that a flag of $I$-graded  subspace of $\mb{V}$
\[
F=(\mb{V}=\mb{F}^0 \supset \mb{F}^1 \supset \cdots \supset \mb{F}^m \supset \mb{F}^{m+1} \supset \cdots \supset \mb{F}^{2m}=\{0\})
\]
is of type $(\mb{i},\mb{a})$ if 
\benu[{\rm(i)}]
\item $\dim(\mb{F}^{\ell-1}/\mb{F}^\ell)_i=\left\{
\begin{array}{ll}
a_\ell & (i=i_\ell) \\
0 & (i \neq i_\ell) 
\end{array}
\right. $,
\item $\mb{F}^{2m-\ell}=(\mb{F}^{\ell})^{\bot}$.
\eenu
Then we have $\wt\mb{V}=\sum_{1 \le \ell \le 2m}a_\ell\alpha_{i_\ell}$. We denote by $\thF{i}{a}$ the set of flags of type $(\mb{i},\mb{a})$. \\
\quad For $x \in \tEO{V}{\Omega}$, a flag $F$ of type $(\mb{i},\mb{a})$ is $x$-stable if $\mb{F}^\ell \ (\ell=1, \ldots ,2m)$ are $x$-stable. We define
\[
\ttFl{i}{a}{\Omega}\colon =\{(x,F) \in \tEO{V}{\Omega} \times \thF{i}{a} \ | \ \text{$F$ is $x$-stable}\}.
\]
The group $\tG{V}$ naturally acts on $\thF{i}{a}$ and $\ttFl{i}{a}{\Omega}$.
\end{dfn}
Note that $x\colon \mb{V} \to \mb{V} \cong \mb{V}^*$ in $\tEO{V}{\Omega}$ may be regarded as a skew-symmetric form on $\mb{V}$, and the condition that $F$ is $x$-stable is equivalent to the one $x(\mb{F}^\ell,\mb{F}^{2m-\ell})=0$ for any $\ell$. \\
\quad The following lemma is obvious.
\begin{lem}\label{lem:sym-proper}
The variety $\ttFl{i}{a}{\Omega}$ is smooth and irreducible. The first projection $\tpi{i}{a}\colon \ttFl{i}{a}{\Omega} \to \tEO{V}{\Omega}$ is $\tG{V} \times (\bb{C}^\times)^{\Omega,\theta}$-equivariant and projective.
\end{lem}

\subsection{Perverse sheaves on $\tEO{V}{\Omega}$}
Let $\Omega$ be a $\theta$-orientation. By Lemma \ref{lem:sym-proper} and Lemma \ref{lem:perv},
\[
\tLL{i}{a}{\Omega}\colon =(\tpi{i}{a})_!(\mb{1}_{\ttFl{i}{a}{\Omega}})
\]
is a semisimple complex in $\ms{D}(\tEO{V}{\Omega})$.
\begin{dfn}
We define $\tmP{V}{\Omega}$ as the set of the isomorphism classes of simple perverse sheaves $L$ in $\ms{D}(\tEO{V}{\Omega})$ satisfying the property: $L$ appears in $\tLL{i}{a}{\Omega}[d]$ as a direct summand for some integer $d$ and $(\mb{i},\mb{a})$. We denote by $\tmQ{V}{\Omega}$ the full subcategory of $\ms{D}(\tEO{V}{\Omega})$ consisting of objects which are isomorphic to finite direct sums of $L[d]$ with $L \in \tmP{V}{\Omega}$ and $d \in \bb{Z}$. 
\end{dfn}
Note that any object in $\tmQ{V}{\Omega}$ is $\tG{V} \times (\bb{C}^\times)^{\Omega,\theta}$-equivariant. \\
\subsection{Multiplications and Restrictions}
Fix $\theta$-symmetric and $I$-graded vector spaces $\mb{V}$ and $\mb{W}$, and an $I$-graded vector space $\mb{T}$ such that $\wt(\mb{V})=\wt(\mb{W})+\wt(\mb{T})+\theta(\wt(\mb{T}))$. \\
\quad We consider the following diagram
\[
\xymatrix{
\EO{T}{\Omega} \times \tEO{W}{\Omega} & \tpE{\Omega}\ar[l]_(.3){p_1}\ar[r]^{p_2} & \tppE{\Omega}\ar[r]^{p_3} & \tEO{V}{\Omega}
}.
\]
Here $\tppE{\Omega}$ is the variety of $(x,V)$ where $x \in \tEO{V}{\Omega}$ and $V$ is an $x$-stable $I$-graded subspace of $\mb{V}$ such that $V \supset V^\bot$ and $\wt(\mb{V}/V)=\wt(\mb{T})$, and we denote by $\tpE{\Omega}$ the variety of $(x,V,\varphi^{\mb{W}},\varphi^{\mb{T}})$ where $(x,V) \in \tppE{\Omega}$, $\varphi^{\mb{W}}\colon \mb{W} \isoto V/V^{\bot}$ is an isomorphism of $\theta$-symmetric $I$-graded vector spaces and $\varphi^{\mb{T}}\colon \mb{T} \isoto \mb{V}/V$ is an isormorphism of $I$-graded vector spaces.
We define $p_1,p_2$ and $p_3$ by $p_1(x,V,\varphi^\mb{W},\varphi^\mb{T})=(x^\mb{T},x^\mb{W})$, $p_2(x,V,\varphi^\mb{W},\varphi^\mb{T})=(x,V)$ and $p_3(x,V)=x$. Here the morphism $x^\mb{W},x^\mb{T}$ are defined by 
\[
x^\mb{W}_h={\varphi^\mb{W}_{\vi(h)}}^{-1} \circ (x|_{V/V^\bot})_h \circ \varphi^\mb{W}_{\vo(h)}, \quad 
x^\mb{T}_h={\varphi^\mb{T}_{\vi(h)}}^{-1} \circ (x|_{\mb{V}/V})_h \circ \varphi^\mb{T}_{\vo(h)}.
\]
\quad Then $p_1$ is smooth with connected fibers, $p_2$ is a principal $\mb{G}_{\mb{T}} \times \tG{W}$-bundle and $p_3$ is projective.\\
\quad For a $\mb{G}_{\mb{T}}$-equivariant semisimple object $K_\mb{T} \in \ms{Q}_{\mb{T},\Omega}$ and a $\tG{W}$-equivariant semisimple object $K_\mb{W} \in \tmQ{W}{\Omega}$, there exists a unique semisimple object $K'' \in \ms{D}(\tppE{\Omega})$ satisfying $p_1^*(K_\mb{T} \boxtimes K_\mb{W})=p_2^*K''$.
\begin{dfn}
We define $K_\mb{T}*K_{\mb{W}}\colon =(p_3)_!(K'') \in\ms{D}(\tEO{V}{\Omega})$.
\end{dfn}
Next, we fix an $I$-graded vector space $U$ such that
\[
\mb{V} \supset U \supset U^{\bot} \supset \{0\}.
\]
We also fix an isomorphism $\mb{W}\cong U/U^\bot$ as $\theta$-symmetric $I$-graded vector spaces and an isomorphism $\mb{T}\cong \mb{V}/U$ as $I$-graded vector spaces. We consider the following diagram 
\[
\xymatrix{
\EO{T}{\Omega} \times \tEO{W}{\Omega} & \tEE{W,V}{\Omega}\ar[l]_(.45){p}\ar@{^{(}->}[r]^(.6){\iota} & \tEO{V}{\Omega}
}
\]
where 
\[
\tEE{W,V}{\Omega}=\{x \in \tEO{V}{\Omega} \ | \ \text{$U$ is $x$-stable}\}
\]
and $p(x)=(x^\mb{T},x^\mb{W})$, $\iota(x)=x$.
\begin{dfn}
For $K \in \ms{D}(\tEO{V}{\Omega})$, we define $\Res_{\mb{T},\mb{W}}(K)\colon =p_!\iota^*(K)$. 
\end{dfn}
\begin{prop}\label{prop-ef} \ Let $\mb{V}$ and $\mb{W}$ be $\theta$-symmetric $I$-graded vector spaces such that $\wt{\mb{V}}=\wt{\mb{W}}+\alpha_i+\alpha_{\theta(i)}$. For $a \in \bb{Z}_{\ge 0}$, let $\mb{S}_i^a$ be an $I$-graded vector space such that $\wt(\mb{S}_i^a)=a\alpha_i$.
\begin{enumerate}[{\rm(i)}]
\item Suppose $\tLL{i}{a}{\Omega} \in \ms{D}(\tEO{W}{\Omega})$. We have 
\[
\mb{1}_{\mb{S}_i^a}*\tLL{i}{a}{\Omega}=L_{(i,\mb{i},\theta(i)),(a,\mb{a},a)}.
\]
for $a \in \bb{Z}_{\ge 0}$. 
\item Suppose $\tLL{i}{a}{\Omega} \in \ms{D}(\tEO{V}{\Omega})$ and $a_\ell>0$ for all $\ell$ such that $i_\ell=i$. For $1 \le k \le 2m$ such that $i_k=i$, we define $\mb{a}^{(k)}=(a_1^{(k)}, \cdots ,a_{2m}^{(k)})$ by $a_\ell^{(k)}=a_\ell-\delta_{\ell,k}-\delta_{\ell,2m-k+1}$ and we set
\[
M_k(\mb{i},\mb{a}^{(k)})=
\sum_{i_\ell=i,\ell<k}a_\ell^{(k)}+\sum_{k<\ell,h \in \Omega;\vo(h)=i,\vi(h)=i_\ell}a_\ell^{(k)}.
\]
Then we have 
\[
\Res_{\mb{S}_i,\mb{W}}(\tLL{i}{a}{\Omega})=\bigoplus_{i_k=i}{}^\theta\!{L}_{\mb{i},\mb{a}^{(k)};\Omega}[-2 {M}_k(\mb{i},\mb{a}^{(k)})].
\]
\end{enumerate}
\end{prop}
\begin{proof}
(1) \ We consider the following diagram:
\[
\xymatrix{
\ttFl{i}{a}{\Omega}\ar[d]_{\tpi{i}{a}}\ar@{}[dr]|\Box & {}^\theta\!\tilde{E}\ar[d]_{\rho'}\ar[l]_{p'_1}\ar[r]^(.3){p'_2}\ar@{}[dr]|\Box  & {}^\theta\!\widetilde{\cal{F}}_{(i,\mb{i},\theta(i)),(a,\mb{a},a);\Omega}\ar[d]_{\rho''}\ar[dr]^{\pi_{(i,\mb{i},\theta(i)),(a,\mb{a},a)}} & \\
\tEO{W}{\Omega} & \tpE{\Omega}\ar[l]_(.4){p_1}\ar[r]^{p_2} & \tppE{\Omega}\ar[r]^{p_3} & \tEO{V}{\Omega}
}
\]
where 
\[
{}^\theta\!\tilde{E}\colon =\{(x,F,\varphi^\mb{W}) \ | \ (x,F) \in {}^\theta\widetilde{\cal{F}}_{(i,\mb{i},\theta(i)),(a,\mb{a},a)}, \varphi^\mb{W}\colon \mb{W} \cong \mb{F}^1/\mb{F}^{2m+1}\}.
\]
Here, $\rho''\colon {}^\theta\!\widetilde{\cal{F}}_{(i,\mb{i},\theta(i)),(a,\mb{a},a)} \to \tppE{\Omega}$ is given by $(x,F) \mapsto (x,\mb{F}^1)$. Then $\rho''$ is projective and $p_3 \circ \rho''={}^\theta\pi_{(i,\mb{i},\theta(i)),(a,\mb{a},a)}$. Hence $\rho''_!(\mb{1}_{{}^\theta\!\widetilde{\cal{F}}_{(i,\mb{i},\theta(i)),(a,\mb{a},a)}})$ is semisimple and ${}^\theta{L}_{(i,\mb{i},\theta(i)),(a,\mb{a},a);\Omega}=(\pi_{(i,\mb{i},\theta(i)),(a,\mb{a},a)})_!(\mb{1}_{{}^\theta\!\widetilde{\cal{F}}_{(i,\mb{i},\theta(i)),(a,\mb{a},a)}})=(p_3)_!(\rho'')_!(\mb{1}_{{}^\theta\!\widetilde{\cal{F}}_{(i,\mb{i},\theta(i)),(a,\mb{a},a)}})$. 
On the other hand, we have
\[
p_2^*(\rho''_!\mb{1}_{{}^\theta\!\widetilde{\cal{F}}_{(i,\mb{i},\theta(i)),(a,\mb{a},a)}})=\rho'_!(p'_2)^*\mb{1}_{{}^\theta\!\widetilde{\cal{F}}_{(i,\mb{i},\theta(i)),(a,\mb{a},a)}}=\rho'_!(p'_1)^*\mb{1}_{\ttFl{i}{a}{\Omega}}=p_1^*(\tpi{i}{a})_!\mb{1}_{\ttFl{i}{a}{\Omega}}=p_1^*(\tLL{i}{a}{\Omega}).
\]
Hence we have $\mb{1}_{\mb{S}_i^a}*\tLL{i}{a}{\Omega}=(p_3)_!\rho''_!(\mb{1}_{{}^\theta\!\widetilde{\cal{F}}_{(i,\mb{i},\theta(i)),(a,\mb{a},a)}})={}^\theta\!L_{(i,\mb{i},\theta(i)),(a,\mb{a},a)}$. \\
(2) \ Set ${}^\theta\!\widetilde{\cal{F}}(\mb{W},\mb{V})=\{(x,F) \in \ttFl{i}{a}{\Omega} \ | \ \text{$U$ is $x$-stable}\}$ and ${}^\theta\!\mathcal{F}_{\mb{i},\mb{a};\Omega}^{(k)}=\{F \in {}^\theta\!\mathcal{F}_{\mb{i},\mb{a};\Omega} \ | \ \mb{F}^k \subset U, \mb{F}^{k-1} \not\subset U\}$. We define
\[
{}^\theta\!\widetilde{\cal{F}}_k(\mb{W},\mb{V})\colon =
\{(x,F)\in {}^\theta\!\widetilde{\cal{F}}(\mb{W},\mb{V}) \ | \ F \in {}^\theta\!\mathcal{F}_{\mb{i},\mb{a};\Omega}^{(k)}\}.
\]
Then the locally closed smooth subvarieties ${}^\theta\!\widetilde{\cal{F}}_k(\mb{W},\mb{V}) \ (1 \le k \le 2m, i_k=i)$ give a partition ${}^\theta\!\widetilde{\cal{F}}(\mb{W},\mb{V})$.  \\
\quad For a flag $F$ of $\mb{V}$, we define the flag $F|_{U/U^{\bot}}$ by 
\[
F|_{U/U^\bot}=(U/U^\bot=(\mb{F}^0 \cap U)/(\mb{F}^0 \cap U^\bot) \supset \cdots \supset (\mb{F}^{2m} \cap U)/(\mb{F}^{2m} \cap U^\bot)=\{0\}).
\]
Note that for $(x,F) \in {}^\theta\!\widetilde{\cal{F}}_k(\mb{W},\mb{V})$,
\begin{eqnarray*}
\dim(\mb{F}^\ell_j \cap U_j)&=&\dim\mb{F}^{\ell}_j-\delta(j=i,\ell<k), \\
\dim(\mb{F}^\ell_j \cap (U^\bot)_j)&=&\delta(2m-\ell \ge k,j=\theta(i)).
\end{eqnarray*}
We have
\[
\dim((F|_{U/U^\bot})^{\ell-1}/(F|_{U/U^\bot})^\ell)_j=\dim(\mb{F}^{\ell-1}/\mb{F}^{\ell})_j-\delta(j=i,\ell=k)-\delta(j=\theta(i),2m-\ell=k-1).
\]
Hence the flag $F|_{U/U^\bot}$ is a flag of type $(\mb{i},\mb{a}^{(k)})$. Therefore $(x,F) \mapsto (x|_{U/U^\bot},F|_{U/U^\bot})$ defines $f_{\mb{a}^{(k)}}\colon {}^\theta\!\widetilde{\cal{F}}_k(\mb{W},\mb{V}) \to {}^\theta\!\widetilde{\cal{F}}_{\mb{i},\mb{a}^{(k)},\Omega}$. We obtain the following diagram: 
\[
\xymatrix{
{}^\theta\!\widetilde{\cal{F}}_{\mb{i},\mb{a}^{(k)},\Omega}\ar[d]_{{}^\theta\!\pi_{\mb{i},\mb{a}^{(k)}}} & {}^\theta\!\widetilde{\cal{F}}_k(\mb{W},\mb{V})\ar[l]_{f_{\mb{a}^{(k)}}}\ar@{^{(}->}[r] & {}^\theta\!\widetilde{\cal{F}}(\mb{W},\mb{V})\ar@{^{(}->}[r]\ar[d] & \ttFl{i}{a}{\Omega}\ar[d]^{\tpi{i}{a}} \\
\tEO{W}{\Omega} & & \tEE{W,V}{\Omega}\ar[ll]_(.4){p}\ar@{^{(}->}[r]^(.6){\iota} & \tEO{V}{\Omega}
}
\]
\begin{cl*}
The morphism $f_{\mb{a}^{(k)}}$ is an affine bundle of rank $M_k(\mb{i},\mb{a}^{(k)})$.
\end{cl*}
\begin{proof}
Fix $(x_\mb{W},F_\mb{W}) \in {}^\theta\!\widetilde{\cal{F}}_{\mb{i},\mb{a}^{(k)},\Omega}$. Note that $(U^\bot)_j=\{0\}$ and $U_j \cong \mb{W}_j$ for $j \neq \theta(i)$. If $F \in {}^\theta\!\mathcal{F}_{\mb{i},\mb{a};\Omega}^{(k)}$ satisfies $F|_{U/U^\bot}=F_\mb{W}$, we have
\begin{eqnarray*}
\mb{F}^\ell_i=\mb{F}^\ell_{\mb{W},i} \ (\ell \ge k), \quad 
\mb{F}^\ell_i=\mb{F}^\ell_{\mb{W},i}+\mb{F}^{k-1}_i \ (\ell<k), \quad
\mb{F}^\ell_{\theta(i)}=(\mb{F}_i^{2m-k+1})^\bot
\end{eqnarray*}
and $\mb{F}^\ell_j=\mb{F}^\ell_{\mb{W},j} \ (j \neq i,\theta(i))$. A subspace $\mb{F}^{k-1}_i$ is parametrized by a one-dimensional subspace $\mb{F}^{k-1}_i/\mb{F}^{k-1}_{\mb{W},i} \subset \mb{V}_i/\mb{F}^{k-1}_{\mb{W},i}$ such that $\mb{F}^{k-1}_i/\mb{F}^{k-1}_{\mb{W},i} \not\subset U_i/\mb{F}^{k-1}_{\mb{W},i}$. 
Hence the fibers of ${}^\theta\!{\cal{F}}_{\mb{i},\mb{a},\Omega}^{(k)} \to {}^\theta\!{\cal{F}}_{\mb{i},\mb{a}^{(k)},\Omega}\colon F \mapsto F|_{U/U^\bot}$ at $F_\mb{W}$ is isomorphic to $\bb{A}^{\dim(\mb{V}_i/\mb{F}^{k-1}_{\mb{W},i})-1}$. Note that 
\[
\dim(\mb{V}_i/\mb{F}^{k-1}_{\mb{W},i})-1=\sum_{\ell<k,i_\ell=i}a_\ell=\sum_{\ell<k,i_\ell=i}a_\ell^{(k)}.
\]
\quad Fix a flag $F \in {}^\theta\!\mathcal{F}_{\mb{i},\mb{a};\Omega}^{(k)}$ such that $F|_{U/U^\bot}=F_\mb{W}$. Note that $\mb{V}_i \supset U_i \cong \mb{W}_i$, $\mb{V}_{\theta(i)}=U_{\theta(i)}$ and $\mb{V}_{j}=U_{j} \cong \mb{W}_{j}$ for $j \neq i,\theta(i)$. Assume that $x \in \tEO{V}{\Omega}$ satisfies the condition that $F$ is $x$-stable and $x|_{U/U^\bot}=x_\mb{W}$. \\
\quad First, suppose that $h \in \Omega$ satisfies $\vo(h) \neq i$ and $\vi(h) \neq \theta(i)$. Then $x_h$ coincides with the composition $\mb{V}_{\vo(h)} \twoheadrightarrow U_{\vo(h)}/(U^\bot)_{\vo(h)} \cong \mb{W}_{\vo(h)} \stackrel{x_{\mb{W},h}}{\longrightarrow} \mb{W}_{\vi(h)} \cong U_{\vi(h)} \subseteq \mb{V}_{\vi(h)}$. Hence, for such an $h \in \Omega$, $x_h$ is uniquely determined by $x_{\mb{W}}$ and $x$ stabilizes the flag $F$. \\
\quad Second, suppose that $h \in \Omega$ satisfies $\vo(h)=i$. Take $v \in \mb{F}_i^{k-1}$ such that $v \notin U_i$. \\
\quad If $\vi(h) \neq \theta(i)$, $x_h$ is parametrised by $x_h(v) \in \mb{F}_{\vi(h)}^{k-1}$. Note that
\[
\dim\mb{F}_{\vi(h)}^{k-1}=\sum_{\ell \ge k,i_\ell=\vi(h)}a_\ell=\sum_{\ell>k,i_\ell=\vi(h)}a_\ell^{(k)},
\]
because $\vi(h) \neq i,\theta(i)$, $i_k=i$ and $\ell \neq k, 2m-k+1$. \\
\quad If $\vo(h)=i$ and $\vi(h)=\theta(i)$, we can regard $x_h$ as a skew-symmetric form on $\mb{V}_i$. Since $\mb{F}^\ell_i=\mb{F}^\ell_{\mb{W},i}+\delta(\ell<k)\bb{C}v$, the skew-symmetric condition on $x$ is equal to the condition $x(v,\mb{F}^{2m-k+1}_i+\bb{C}v)=0$. Then $x_h$ is parametrized by $\left(\mb{V}_i/(\mb{F}_i^{2m-k+1}+\bb{C}v)\right)^*$. Since $v \notin \mb{F}^{2m-k+1}_i$ if and only if $2m-k+1 \ge k$, we have 
\begin{eqnarray*}
&{}&\dim\left(\mb{V}_i/(\mb{F}_i^{2m-k+1}+\bb{C}v)\right)^*=\dim\left(\mb{V}/\mb{F}_i^{2m-k+1}\right)^*-\delta(2m-k+1 \ge k) \\
&=&\left(\dim{F_{\theta(i)}^{k-1}}\right)-\delta(2m-k+1 \ge k)=\left(\sum_{\ell \ge k,i_\ell=\theta(i)}a_\ell\right)-\delta(2m-k+1 \ge k).
\end{eqnarray*}
Since $i_k=i \neq \theta(i)$, $i_{2m-k+1}=\theta(i)$, we have $a_\ell=a_{\ell}^{(k)}+\delta(\ell=2m-k+1)$ if $i_\ell=\theta(i)$. Thus we obtain 
\[
\dim\left(\mb{V}_i/(\mb{F}_i^{2m-k+1}+\bb{C}v)\right)=\sum_{\ell>k,i_\ell=\theta(i)}a_\ell^{(k)}.
\]
\quad Set
\[
\Omega_0:=\{h \in \Omega \ | \ \vo(h)=i,\vi(h)=\theta(i)\}, \quad 
\Omega_1:=\{h \in \Omega \ | \ \vo(h)=i,\vi(h) \neq \theta(i)\}.
\]
\quad The morphism $\widetilde{\cal{F}}_k(\mb{W},\mb{V}) \to \{F \in {}^\theta\!{\mathcal{F}}_{\mb{i},\mb{a},\Omega}^{(k)} \ | \ F|_{U/U^\bot}=F_{\mb{W}}\}$ is an affine bundle and its fiber dimension is equal to 
\begin{eqnarray*}
&{}&\sum_{h \in \Omega_1}\dim(\mb{F}_{\vi(h)}^{k-1})+\sum_{h \in \Omega_0}\dim\{\mb{V}_i/(\mb{F}_i^{2m-k+1}+\bb{C}v)\} \\
&=&\sum_{h \in \Omega_1, \ell>k,i_\ell \neq \theta(i)}a_\ell^{(k)}+\sum_{h \in \Omega_0, \ell>k,i_\ell=\theta(i)}a_\ell^{(k)}=\sum_{h \in \Omega_0 \sqcup \Omega_1,\ell>k}a_\ell^{(k)}.
\end{eqnarray*}

\quad Thus the rank of $f_{\mb{a}^{(k)}}$ is equal to 
\begin{eqnarray*}
\dim(\mb{V}_i/\mb{F}^{k-1}_{\mb{W},i})-1+\sum_{h \in \Omega_0 \sqcup \Omega_1,\ell>k}a_\ell^{(k)}=\sum_{i_\ell=i,\ell<k}a_\ell^{(k)}+\sum_{h \in \Omega_0 \sqcup \Omega_1,k<\ell}a_\ell^{(k)}=M_k(\mb{i},\mb{a}^{(k)}).
\end{eqnarray*}
\end{proof}
By this claim, we have $(f_{\mb{a}^{(k)}})_!\mb{1}_{{}^\theta\!\widetilde{\mathcal{F}}_k(\mb{W},\mb{V})}=\mb{1}_{{}^\theta\!\widetilde{\cal{F}}_{\mb{i},\mb{a}^{(k)},\Omega}}[-2 {M}_k(\mb{i},\mb{a}^{(k)})]$. By Lemma \ref{lem:perv}(2), we obtain
\begin{eqnarray*}
\Res_{\mb{S}_i,\mb{W}}(\tLL{i}{a}{\Omega})&=&({}^\theta\!\pi_{\mb{i},\mb{a};\Omega})_!\mb{1}_{{}^\theta\!\widetilde{\mathcal{F}}_{\mb{i},\mb{a};\Omega}}=\bigoplus_{k}({}^\theta\!\pi_{\mb{i},\mb{a}^{(k)}})_!(f_{\mb{a}^{(k)}})_!\mb{1}_{{}^\theta\!\widetilde{\mathcal{F}}_k(\mb{W},\mb{V})} \\
&=&\bigoplus_{i_k=i}{}^\theta\!{L}_{\mb{i},\mb{a}^{(k)};\Omega}[-2 {M}_k(\mb{i},\mb{a}^{(k)})].
\end{eqnarray*}
\end{proof}

\begin{lem}\label{ass}
Let $\mb{T}^1$ and $\mb{T}^2$ be $I$-graded vector spaces. Let $\mb{W}$ and $\mb{V}$ be $\theta$-symmetric $I$-graded vector spaces such that $\wt{\mb{V}}=\wt{\mb{T}^1}+\theta(\wt{\mb{T}^1})+\wt{\mb{T}^2}+\theta(\wt{\mb{T}^2})+\wt{\mb{W}}$. \\
\quad For $\mb{G}_{\mb{T}^j}$-equivariant semisimple objects $L_j \in \ms{D}(\mb{E}_{\mb{T}^j,\Omega}) \ (j=1,2)$ and a $\tG{W}$-equivariant semisimple obejct $L \in \ms{D}(\tEO{W}{\Omega})$, we have $(L_1*L_2)*L \cong L_1*(L_2*L)$.
\end{lem}
\begin{proof}
Let $\mb{T}^{12}$ be an $I$-graded vector space such that $\wt{\mb{T}^{12}}=\wt{\mb{T}^1}+\wt{\mb{T}^2}$. Let $\mb{W}^2$ be a $\theta$-symmetric $I$-graded vector space such that $\wt{\mb{W}^2}=\wt{\mb{T}^2}+\theta(\wt{\mb{T}^2})+\wt{\mb{W}}$. \\
\quad We denote by $\ms{F}$ the variety of pairs $(x,F)$ where $x \in \tEO{V}{\Omega}$ and $F=(\mb{V} \supset F^1 \supset F^2 \supset F^3 \supset F^4 \supset \{0\}$ is an $x$-stable flag such that $F^3=(F^2)^\bot,F^4=(F^1)^\bot$, $F^1/F^4 \cong \mb{W}^2$ and $F^2/F^3 \cong \mb{W}$ as $\theta$-symmetric $I$-graded vector spaces. Let $\tilde{\ms{F}}$ be the variety of pairs $(x,F,\varphi_{\mb{W}},\varphi_{\mb{W}^2},\varphi_1,\varphi_2,\varphi_{\mb{T}^2})$ where $(x,F) \in \ms{F}$ and $\varphi_{\mb{W}^2}\colon F^1/F^4 \cong \mb{W}^2$, $\varphi_{\mb{W}}\colon F^2/F^3 \cong \mb{W}$ as $\theta$-symmetric $I$-graded vector spaces, and $\varphi_1\colon \mb{V}/F^1 \cong \mb{T}^1$, $\varphi_2\colon \mb{V}/F^2 \cong \mb{T}^{12}$ and $\varphi_{\mb{T}^2}\colon F^1/F^2 \cong \mb{T}^2$ as $I$-graded vector spaces. \\
\quad We consider the folowing diagram:
\[
\xymatrix{
\mb{E}_{\mb{T}^1,\Omega} \times \mb{E}_{\mb{T}^2,\Omega} \times \tEO{W}{\Omega} & \tilde{\ms{F}}\ar[l]_(.2){u_1}\ar[r]^{u_2} & \ms{F}\ar[r]^{u_3} & \tEO{V}{\Omega}.
}
\]
Here $u_1(x,F,\varphi_{\mb{W}},\varphi_{\mb{W}^2},\varphi_1,\varphi_2,\varphi_{\mb{T}^2})=(x^1,x^2,x_{\mb{W}})$, where $x_{\mb{W}},x^1$ and $x^2$ are the restrictions of $x$ to $\mb{W},\mb{T}^1$ and $\mb{T}^2$ through the isomorphism $\varphi_{\mb{W}},\varphi_1$ and $\varphi_2$ respectively, and $u_2$ and $u_3$ are natural projections. Note that $u_1$ is smooth with connected fibers, $u_2$ is a principal $\mb{G}_{\mb{T}^1} \times \mb{G}_{\mb{T}^2} \times \tG{W}$-bundle and $u_3$ is projective. Then, for $L \in \tmP{W}{\Omega}$, there exists a unique semisimple object $L'' \in \ms{D}(\ms{F})$ such that $u_1^*(L_1 \boxtimes L_2 \boxtimes L)=u_2^*L''$, we define $K$ by $(u_3)_!L''$. We shall prove $K \cong L_1*(L_2*L)$ and $K \cong (L_1*L_2)*L$. \\
\quad First, $L_2*L$ is defined by the following diagram
\[
\xymatrix{
\mb{E}_{\mb{T}^2,\Omega} \times \tEO{W}{\Omega} & E'_2\ar[l]_(.25){q_1}\ar[r]^{q_2} & E''_2 \ar[r]^{q_3} & \tEO{W^{\rm 2}}{\Omega}.
}
\]
Here $E''_2$ is the variety of $(y,V)$ where $y \in \tEO{W^{\rm 2}}{\Omega}$ and $V$ is an $y$-stable $I$-graded vector subspace of $\mb{W}^2$ such that $V \supset V^\bot$ and $\wt(\mb{W}^2/V)=\wt(\mb{T}^2)$, and $E'_2$ is the variety of $(y,V,\psi_\mb{W},\psi_{\mb{T}^2})$ where $(y,V) \in E''_2$ and $\psi_\mb{W}:V/V^\bot \cong \mb{W}$ and $\psi_{\mb{T}^2}:\mb{W}^2/V \cong \mb{T}^2$. For $L''_2 \in \ms{D}(E'_2)$ such that $q_1^*(L_2 \boxtimes L)=q_2^*L''_2$, we have $(q_3)_!L''_2=L_2*L$. We consider the diagram
\[
\xymatrix{
\mb{E}_{\mb{T}^1,\Omega} \times \mb{E}_{\mb{T}^2,\Omega} \times \tEO{W}{\Omega} & \mb{E}_{\mb{T}^1,\Omega} \times E'_2\ar[l]_(.35){\tilde{q}_1}\ar[r]^{\tilde{q}_2} & \mb{E}_{\mb{T}^1,\Omega} \times E''_2 \ar[r]^(.4){\tilde{q}_3} & \mb{E}_{\mb{T}^1,\Omega} \times \tEO{W^{\rm 2}}{\Omega},
}
\]
and denote by $L''_1\colon =L_1 \boxtimes L''_2 \in \ms{D}(\mb{E}_{\mb{T}^1,\Omega} \times E''_2)$. Then $\tilde{q}_1^*(L_1 \boxtimes L_2 \boxtimes L)=\tilde{q}_2^*L''_1$ and $(\tilde{q}_3)_!L''_1=L_1 \boxtimes (L_2*L)$. \\
\quad Second, $L_1*(L_2*L)$ is defined by the following diagram:
\[
\xymatrix{
\mb{E}_{\mb{T}^1,\Omega} \times \tEO{W^2}{\Omega} & E'\ar[l]_(.25){p_1}\ar[r]^{p_2} & E'' \ar[r]^(.4){p_3} & \tEO{V}{\Omega}.
}
\]
Here $E''$ is the variety of $(y,V)$ where $y \in \tEO{V}{\Omega}$ and $V$ is an $y$-stable $I$-graded vector subspace of $\mb{V}$ such that $V \supset V^\bot$ and $\wt(\mb{V}/V)=\wt(\mb{T}^1)$, and $E'$ is the variety of $(y,V,\psi_{\mb{W}^2},\psi_{\mb{T}^1})$ where $(y,V) \in E''$ and $\psi_{\mb{W}^2}:V/V^\bot \cong \mb{W}^2$ and $\psi_{\mb{T}^1}:\mb{V}/V \cong \mb{T}^1$. For $K'' \in \ms{D}(E'')$ such that $p_1^*(L_1 \boxtimes (L_2*L))=p_2^*K''$, we have $L_1*(L_2*L)=(p_3)_!K''$. \\
\quad Set $E'_1=\mb{E}_{\mb{T}^1,\Omega} \times E'_2$, $E''_1=\mb{E}_{\mb{T}^1,\Omega} \times E''_2$, $E_{12}=\mb{E}_{\mb{T}^1,\Omega} \times \mb{E}_{\mb{T}^2,\Omega} \times \tEO{W}{\Omega}$ and $E_2=\mb{E}_{\mb{T}^1,\Omega} \times \tEO{W^{\rm 2}}{\Omega}$. We consider the following diagram:
\[
\xymatrix{
 & & & E_2 & & & \\
 & & E''_1\ar[ur]^{\tilde{q}_3} & & E'\ar[ul]_{p_1}\ar[dr]^{p_2} & & \\
 & E'_1\ar[dl]_{\tilde{q}_1}\ar[ur]^{\tilde{q}_2} & & \tilde{E}\ar[ur]^{t_3}\ar[ul]_{r_1}\ar[dr]^{r_2}\ar@{}[rr]|\Box\ar@{}[uu]|\Box & & E''\ar[dr]^{p_3} & \\
E_{12} & & \tilde{\ms{F}}\ar[ll]^{u_1}\ar[rr]_{u_2}\ar[ul]_{v_1}\ar[ur]^{t_2} & & \ms{F}\ar[rr]_{u_3}\ar[ur]^{s_3} & & \tEO{V}{\Omega}
}
\]
where $\tilde{E}=\ms{F} \times_{E''} E'$. Here $s_3(x,F)=(x,F^1)$, $t_2(x,F,\varphi_{\mb{W}},\varphi_{\mb{W}^2},\varphi_1,\varphi_2)=(x,F,\varphi_{\mb{W}^2},\varphi_1)$, $r_2(x,F,\varphi_{\mb{W}^2},\varphi_1)=(x,F)$ and $t_3(x,F,\varphi_{\mb{W}^2},\varphi_1)=(x,F^1,\varphi_{\mb{W}^2},\varphi_1)$. We define $r_1$ and $v_1$ by 
\begin{eqnarray*}
r_1(x,F,\varphi_{\mb{W}^2},\varphi_1)&=&(x^1,x_{\mb{W}^2},\varphi_{\mb{W}^2}(F_2/F_4)), \\
v_1(x,F,\varphi_{\mb{W}},\varphi_{\mb{W}^2},\varphi_1,\varphi_2,\varphi_{\mb{T}^2})&=&(x^1,x_{\mb{W}^2},\varphi_{\mb{W}^2}(F_2/F_4),\psi_{\mb{W}},\psi_2),
\end{eqnarray*}
where $x_{\mb{W}^2},\psi_{\mb{W}}$ and $\psi_2$ are natural morphism induced by using $\varphi_{\mb{W}}, \varphi_{\mb{W}^2}$ and $\varphi_2$. \\
\quad We have $t_2^*r_1^*L''_1=v_1^*\tilde{q}_2^*L''_1=v_1^*\tilde{q}_1^*(L_1 \boxtimes L_2 \boxtimes L)=u_1^*(L_1 \boxtimes L_2 \boxtimes L)=u_2^*L''=t_2^*r_2^*L''$. Since $t_2$ is a $\mb{G}_{\mb{T}^2} \times \tG{W^1}$-principal bundle, we obtain $r_1^*L''_1=r_2^*L''$. Therefore $p_2^*(s_3)_!L''=(t_3)_!r_2^*L''=(t_3)_!r_1^*L''=p_1^*(q_3)_!L''_1=p_1^*(L_1 \boxtimes (L_2*L))$. Thus $(p_3)_!(r_3)_!L''=L_1*(L_2*L)$. We have $K=(u_3)_!L''=L_1*(L_2*L)$. \\
\quad Similarly, we obtain $K \cong (L_1*L_2)*L$. Thus the claim follows.
\end{proof}
\subsection{Restriction functor $E_i$, Induction functors $F_i$ and $F_i^{(a)}$}
\quad We consider the following diagram
\[
\xymatrix{
\EO{T}{\Omega} \times \tEO{W}{\Omega} & \tpE{\Omega}\ar[l]_(.3){p_1}\ar[r]^{p_2} & \tppE{\Omega}\ar[r]^{p_3} & \tEO{V}{\Omega}
}.
\]
\begin{lem}
Suppose $\mb{T}=\mb{S}_i$. Let $d_{p_1}$ and $d_{p_2}$ be the dimension of the fibers of $p_1$ and $p_2$, respectively. The we have
\[
d_{p_1}-d_{p_2}=\dim{\tppE{\Omega}}-\dim{\tEO{W}{\Omega}}=\dim{\mb{W}}_i+\sum_{h \in \Omega\colon \vo(h)=i}\dim{\mb{W}_{\vi(h)}}.
\]
\end{lem}
\begin{proof}
For a vector space $V$, we denote by $\Alt(V)$ the set of all skew-symmetric linear maps $V \to V^*$. Let $\mb{P}(V)$ denote the projective space of hyperplanes of $V$. Set $\Omega_0=\{h \in \Omega \ | \ \theta(h)=h \}, \Omega_1=\Omega \backslash \Omega_0$. We have
\[
\dim{\tEO{W}{\Omega}}=\dfrac{1}{2}\sum_{h \in \Omega_1}{\dim\mb{W}_{\vo{(h)}}\dim\mb{W}_{\vi(h)}}+\sum_{h \in \Omega_0}{\dim\Alt(\mb{W}_{\vo(h)})}.
\]
We set
\begin{eqnarray*}
\Omega_{10}&=&\{h \in \Omega_1 \ | \ \vo(h) \neq i,\vi(h) \neq i\}, \\
\Omega_{11}&=&\{h \in \Omega_1 \ | \ \vo(h)=i\}, \\
\Omega_{12}&=&\{h \in \Omega_1 \ | \ \vi(h)=i\}, \\
\Omega_{00}&=&\{h \in \Omega_0 \ | \ \text{$(\vo(h),\vi(h))=(i,\theta(i))$ or $(\theta(i),i)$}\}, \\
\Omega_{01}&=&\Omega_0 \backslash \Omega_{00}.
\end{eqnarray*}
Then $\Omega_1=\Omega_{10} \sqcup \Omega_{11} \sqcup \Omega_{12}$ and $\Omega_0=\Omega_{00} \sqcup \Omega_{01}$. Note that $\theta$ gives bijections $\Omega_{10} \to \Omega_{10}$ and $\Omega_{11} \to \Omega_{12}$. Therefore we have
\begin{eqnarray*}
\dim\tppE{\Omega}&=&\dim\mb{P}(\mb{V}_i)+\dfrac{1}{2}\sum_{h \in \Omega_{10}}{\dim\mb{W}_{\vo{(h)}}\dim\mb{W}_{\vi(h)}}\\
&{}&+\sum_{h \in \Omega_{11}}\dim\mb{V}_i\dim\mb{W}_{\vi(h)}+\sum_{h \in \Omega_{01}}\dim\Alt(\mb{W}_{\vo(h)})\\
&{}&+\sum_{h \in \Omega,\vo(h)=i,\vi(h)=\theta(i)}\dim\Alt(\mb{V}_i)+\sum_{h \in \Omega,\vo(h)=\theta(i),\vi(h)=i}\dim\Alt(\mb{W}_i).
\end{eqnarray*}
Since $\dim\mb{V}_i=\dim\mb{W}_i+1$ and $\dim\Alt(\mb{V}_i)-\dim\Alt(\mb{W}_i)=\dim\mb{W}_i$, we conclude
\begin{eqnarray*}
&{}&\dim{\tppE{\Omega}}-\dim{\tEO{W}{\Omega}} \\
&=&\dim{\mb{W}}_i+\sum_{h \in \Omega_{11}}\dim\mb{W}_{\vi(h)}+\sum_{h \in \Omega,\vo(h)=i,\vi(h)=\theta(i)}(\dim\Alt(\mb{V}_i)-\dim\Alt(\mb{W}_i)) \\
&=&\dim{\mb{W}}_i+\sum_{h \in \Omega,\vo(h)=i,\vi(h)\neq\theta(i)}\dim\mb{W}_{\vi(h)}+\sum_{h \in \Omega,\vo(h)=i,\vi(h)=\theta(i)}\dim\mb{W}_i \\
&=&\dim{\mb{W}}_i+\sum_{h \in \Omega\colon \vo(h)=i}\dim{\mb{W}_{\vi(h)}}.
\end{eqnarray*}
\end{proof}

\begin{dfn} \quad
\begin{enumerate}[(i)]
\item For $\mb{T}=\mb{S}_i$ and a $\tG{W}$-equivariant semisimple object $K$ in $\tmQ{W}{\Omega}$, we define the operator $F_i$ by 
\[
F_i(K)\colon =(\mb{1}_{\mb{S}_i}*K)\left[d_{F_i}\right]
\]
where 
\[
d_{F_i}=d_{p_1}-d_{p_2}=\dim{\mb{W}}_i+\sum_{h \in \Omega\colon \vo(h)=i}\dim{\mb{W}_{\vi(h)}}.
\]
\item For $\mb{T}=\mb{S}_i$, we define the functor $E_i\colon \ms{D}(\tEO{V}{\Omega}) \to \ms{D}(\tEO{W}{\Omega})$ by 
\[
E_i(K)\colon =\Res_{\mb{S}_i,\mb{W}}(K)\left[d_{E_i}\right]
\]
where
\[
d_{E_i}=d_{F_i}-2\dim{\mb{W}_i}=-\dim{\mb{W}_i}+\sum_{h \in \Omega\colon \vo(h)=i}\dim{\mb{W}_{\vi(h)}}.
\]
\end{enumerate}
\end{dfn}
By Prposition \ref{prop-ef}, $E_i$ and $F_i$ induce the restriction functor $\tmQ{V}{\Omega} \to \tmQ{W}{\Omega}$, induction functor $\tmQ{W}{\Omega} \to \tmQ{V}{\Omega}$, respectively.\\

\begin{dfn}
For $a \in \bb{Z}_{>0}$, let $\mb{W}$ and $\mb{V}$ be $\theta$-symmetric $I$-graded vector spaces such that $\wt(\mb{V})=\wt(\mb{W})+a(\alpha_i+\alpha_{\theta(i)})$.  For a $\tG{W}$-equivariant semisimple object $L \in \tmP{W}{\Omega}$, we define $F_i^{(a)}(L)\colon =\mb{1}_{\mb{S}_i^a}*L[d_a]$ where 
\[
d_a=a\left(\dim{\mb{W}_i}+\sum_{h \in \Omega\colon \vo(h)=i}\dim\mb{W}_{\vi(h)}\right)+\dfrac{a(a-1)}{2}\sharp\{h \in \Omega|\vo(h)=i,\vi(h)=\theta(i)\}.
\]
We call $F_i^{(a)}$ the $a$-th divided power of $F_i$.
\end{dfn}

By Proposition \ref{prop-ef}(1), we have the following lemma.
\begin{lem}\label{cor-div}
The object $\tLL{i}{a}{\Omega}$ is isomorphic to $F_{i_1}^{(a_1)}F_{i_2}^{(a_2)} \cdots F_{i_{m}}^{(a_{m})}\mb{1}_{\pt}$ up to shift.
\end{lem}

\begin{lem}\label{lem-df}
The operator $F_i^{(a)}$ gives a functor $\tmQ{W}{\Omega} \to \tmQ{V}{\Omega}$ and satisfy $F_iF_i^{(a)}=F_i^{(a)}F_i=[a+1]_vF_i^{(a+1)}$.
\end{lem}
\begin{proof}
By Proposition \ref{prop-ef}(1), $F_i^{(a)}$ gives a functor $\tmQ{W}{\Omega} \to \tmQ{V}{\Omega}$. We have 
\[
F_iF_i^{(a)}(L)=F_i(\mb{1}_{\mb{S}_i^a}*L)[d_a]=\mb{1}_{\mb{S}_i}*(\mb{1}_{\mb{S}_i^a}*L)v^{d_a+d}
\]
where
\begin{eqnarray*}
d&=&\dim\mb{W}_i+a+\sum_{h \in \Omega\colon \vo(h)=i,\vi(h)\neq\theta(i)}\dim\mb{W}_{\vi(h)}+\sum_{h \in \Omega\colon \vo(h)=i,\vi(h)=\theta(i)}(\dim\mb{W}_{\theta(i)}+a) \\
&=&\dim\mb{W}_i+a+\sum_{h \in \Omega\colon \vo(h)=i}\dim\mb{W}_{\vi(h)}+a\sharp\{h \in \Omega|\vo(h)=i,\vi(h)=\theta(i)\}.
\end{eqnarray*}
Note that $\mb{1}_{\mb{S}_i}*\mb{1}_{\mb{S}_i^a}=(1+v^{-2}+ \cdots +v^{-2a})\mb{1}_{\mb{S}_i^{a+1}}=[a+1]_vv^{-a}\mb{1}_{\mb{S}_i^{a+1}}$ in $\mb{E}_{\mb{S}_i^{a+1},\Omega}$. By Lemma \ref{ass}, we have
\begin{eqnarray*}
F_iF_i^{(a)}(L)&=&[a+1]_vv^{-a}v^{d_a+d}\mb{1}_{\mb{S}_i^{a+1}}*L \\
&=&[a+1]_vv^{-a}v^{d_a+d-d_{a+1}}F_i^{(a+1)}(L).
\end{eqnarray*}
Since 
\begin{eqnarray*}
d_a+d&=&(a+1)\left(\dim{\mb{W}_i}+\sum_{h \in \Omega\colon \vo(h)=i}\dim\mb{W}_{\vi(h)}\right)+a \\
&{}& \quad +\left(\dfrac{a(a-1)}{2}+a\right)\sharp\{h \in \Omega|\vo(h)=i,\vi(h)=\theta(i)\} \\
&=&d_{a+1}+a,
\end{eqnarray*}
we conculde $F_iF_i^{(a)}=[a+1]_vF_i^{(a+1)}$. 
\end{proof}

\subsection{Commutativity with Fourier transforms}
For two $\theta$-orientations $\Omega$ and $\Omega'$, we have $\overline{\Omega \backslash \Omega'}=\Omega' \backslash \Omega$. Then we can regard $\tEO{V}{\Omega} \to \tEO{V}{\Omega \cap \Omega'}$ and $\tEO{V}{\Omega'} \to \tEO{V}{\Omega \cap \Omega'}$ as vector bundles and they are the dual vector bundle to each other by the form $\sum_{h \in \Omega \backslash \Omega'}\tr(x_hx_{\bar{h}})$ on $\tEO{V}{\Omega} \times \tEO{V}{\Omega'}$. We say that $L \in \ms{D}(\tEO{V}{\Omega})$ is $(\bb{C}^\times)^{\Omega,\theta}$-monodromic if $H^j(L)$ is locally constant on every $(\bb{C}^\times)^{\Omega,\theta}$-orbit on $\tEO{V}{\Omega}$. Let $\ms{D}_{(\bb{C}^\times)^{\Omega,\theta}\rm{-mono}}(\tEO{V}{\Omega})$ be the full subcategory of $\ms{D}(\tEO{V}{\Omega})$ consisting of $(\bb{C}^\times)^{\Omega,\theta}$-monodromic objects. Hence we have the Fourier transform $\Phi_{\mb{V}}^{\Omega\Omega'}\colon \ms{D}_{(\bb{C}^\times)^{\Omega,\theta}\rm{-mono}}(\tEO{V}{\Omega}) \to \ms{D}_{(\bb{C}^\times)^{\Omega,\theta}\rm{-mono}}(\tEO{V}{\Omega'})$. The following lemma is obvious.
\begin{lem}\label{FD-lem}
For three $\theta$-orientations $\Omega,\Omega'$ and $\Omega''$, we have 
\[
\Phi_{\mb{V}}^{\Omega'\Omega''}\circ\Phi_{\mb{V}}^{\Omega\Omega'}\cong a^* \circ \Phi_{\mb{V}}^{\Omega\Omega''} \colon \ms{D}_{(\bb{C}^\times)^{\Omega,\theta}\rm{-mono}}(\tEO{V}{\Omega}) \to \ms{D}_{(\bb{C}^\times)^{\Omega,\theta}\rm{-mono}}(\tEO{V}{\Omega''})
\]
where $a:\tEO{V}{\Omega''} \to \tEO{V}{\Omega''}$ is defined by $x_h \mapsto -x_h$ or $x_h$ according that $h \in \Omega'' \cap \overline{\Omega'} \cap \Omega$ or not. 
In particular, $\ms{D}_{(\bb{C}^\times)^{\Omega,\theta}\rm{-mono}}(\tEO{V}{\Omega})$ does not depend on $\Omega$.
\end{lem}

Since any object in $\tmQ{V}{\Omega}$ is $\tG{V} \times (\bb{C}^\times)^{\Omega,\theta}$-equivariant, it is a monodromic object. By the commutativity between $E_i, F_i$ and $(\bb{C}^\times)^{\Omega,\theta}$-action, the functors $E_i$ and $F_i$ preserve the category $(\bb{C}^\times)^{\Omega,\theta}$-monodromic objects. \\

\begin{thm}\label{com-FD} \ Let $\mb{V}$ and $\mb{W}$ be $\theta$-symmetric $I$-graded vector spaces such that $\wt{\mb{V}}=\wt{\mb{W}}+\alpha_i+\alpha_{\theta(i)}$, and $\Omega$ and $\Omega'$ be two $\theta$-symmetric orientations.
\benu[{\rm(1)}]
\item  Let $F_i^\Omega$ and $F_i^{\Omega'}$ be the induction functors with respect to $\Omega$ and $\Omega'$, respectively. For a $\tG{W}$-equivariant semisimple obejct $L \in \tmQ{W}{\Omega}$, we have $\Phi_{\mb{V}}^{\Omega\Omega'} \circ F_i^{\Omega}(L) \cong F_i^{\Omega'} \circ \Phi_{\mb{W}}^{\Omega\Omega'}(L)$.
\item Let ${E_i}^{\Omega}$ and ${E_i}^{\Omega'}$ be the restriction functors with respect to $\Omega$ and $\Omega'$, respectively. For a $\tG{V}$-equivariant semisimple obejct $K \in \tmQ{W}{\Omega}$, we have $\Phi_{\mb{W}}^{\Omega\Omega'} \circ {E_i}^{\Omega}(K) \cong {E_i}^{\Omega'} \circ \Phi_{\mb{V}}^{\Omega\Omega'}(K)$.
\item The Fourier transform $\Phi_{\mb{V}}^{\Omega\Omega'}$ gives an isomorphism between $\tmP{V}{\Omega}$ and $\tmP{V}{\Omega'}$ and an equivalence between $\tmQ{V}{\Omega}$ and $\tmQ{V}{\Omega'}$.
\eenu
\end{thm}
\begin{proof}
(1) \ Let us define the fibre products $E_1,E_2,E_3,E'_1,E'_2$ and $E'_3$ by
\[
\begin{array}{ll}
E_1\colon =\tEO{W}{\Omega} \times_{\tEO{W}{\Omega \cap \Omega'}}\tpE{\Omega \cap \Omega'}, & E'_1\colon =\tEO{W}{\Omega'} \times_{\tEO{W}{\Omega \cap \Omega'}}\tpE{\Omega \cap \Omega'}, \\
E_2\colon =\tpE{\Omega \cap \Omega'} \times_{\tEO{V}{\Omega \cap \Omega'}}\tEO{V}{\Omega}, & E'_2\colon =\tpE{\Omega \cap \Omega'} \times_{\tEO{V}{\Omega \cap \Omega'}}\tEO{V}{\Omega}, \\
E_3\colon =\tppE{\Omega \cap \Omega'} \times_{\tEO{V}{\Omega \cap \Omega'}}\tEO{V}{\Omega},  & E'_3\colon =\tppE{\Omega \cap \Omega'} \times_{\tEO{V}{\Omega \cap \Omega'}}\tEO{V}{\Omega'}.
\end{array}
\]
Note that $E'_1$ and $E'_2$ are the dual vector bundle of $E_1$ and $E_2$ over $\tpE{\Omega \cap \Omega'}$ respectively, and $E'_3$ is the dual vector bundle of $E_3$ over $\tppE{\Omega \cap \Omega'}$. We denote by $\Phi_{E_j}\colon \ms{D}_{(\bb{C}^\times)^{\Omega,\theta}\rm{-mono}}(E_j) \to \ms{D}_{(\bb{C}^\times)^{\Omega,\theta}\rm{-mono}}(E'_j) \ (j=1,2,3)$ and $\Phi'\colon \ms{D}_{(\bb{C}^\times)^{\Omega,\theta}\rm{-mono}}(\tpE{\Omega}) \to \ms{D}_{(\bb{C}^\times)^{\Omega,\theta}\rm{-mono}}(\tpE{\Omega}^*)$ the Fourier transforms. For simplicity, we denote by $\Phi_\mb{V},\Phi_{\mb{W}}$ insted of $\Phi_{\mb{V}}^{\Omega\Omega'},\Phi_{\mb{W}}^{\Omega\Omega'}$, respectively.\\
\quad We denote by $u_1$ and $u'_1$ the projections $E_1 \to \tEO{W}{\Omega}$ and $E'_1 \to \tEO{W}{\Omega'}$, respectively. Let $\tilde{p_1},\tilde{p'_1},\iota_2$ and $\iota'_2$ be the canonical maps $\tpE{\Omega} \to E_1, \tpE{\Omega'} \to E'_1, \tpE{\Omega} \to E_2$ and $\tpE{\Omega'} \to E'_2$, respectively. Then we obtain the following Cartesian diagram of the vector bundles on $\tpE{\Omega \cap \Omega'}$: 
\[
\xymatrix{
\tpE{\Omega'}\ar[r]^{\iota'_2}\ar[d]_{\tilde{p'_1}}\ar@{}[dr]|\Box & E'_2\ar[d]_{{}^t\iota_2} \\
E'_1\ar[r]^{{}^t\tilde{p_1}} & (\tpE{\Omega})^*
}
\]
Moreover let $u_3$ and $u'_3$ be the projections $E_3 \to \tEO{V}{\Omega}$ and $E'_3 \to \tEO{V}{\Omega'}$, respectively, $\tilde{p_2},\tilde{p'_2},\iota_3$ and $\iota'_3$ the canonical maps $\tppE{\omega} \to E_3$ and $\tppE{\Omega'} \to E'_3$, repectively. We obtain the following diagram:
\[
\xymatrix{
E_1\ar[d]_{u_1} & E_2\ar[r]_{\widetilde{p_2}} & E_3\ar[dr]^{u_3} & \\
\tEO{W}{\Omega}\ar[d] & \tpE{\Omega}\ar[l]_{p_1}\ar[r]^{p_2}\ar[d]\ar[ul]_{\widetilde{p_1}}\ar[u]_{\iota_2}\ar@{}[ur]|\Box & \tppE{\Omega}\ar[r]^{p_3}\ar[d]\ar[u]_{\iota_3} & \tEO{V}{\Omega}\ar[d] \\
\tEO{W}{\Omega \cap \Omega'} & \tpE{\Omega \cap \Omega'}\ar[l]\ar[r] & \tppE{\Omega \cap \Omega'}\ar[r] & \tEO{V}{\Omega \cap \Omega'} \\
\tEO{W}{\Omega'}\ar[u] & \tpE{\Omega'}\ar[l]_{p'_1}\ar[r]^{p'_2}\ar[u]\ar[dl]_{\widetilde{p'_1}}\ar[d]_{\iota'_2}\ar@{}[dr]|\Box & \tppE{\Omega'}\ar[r]^{p'_3}\ar[u]\ar[d]_{\iota'_3} & \tEO{V}{\Omega'}\ar[u] \\
E'_1\ar[u]_{u'_1}\ar[d]_{{}^t\widetilde{p_1}}\ar@{}[r]|\Box & E'_2\ar[r]_{\widetilde{p'_2}}\ar[dl]_{{}^t\iota_2} & E'_3\ar[ur]_{u'_3} & \\
 (\tpE{\Omega})^* & & & 
}
\]
\quad Let $L$ be $\tG{W}$-equivariant semisimple complex on $\tEO{W}{\Omega}$, $L''$ a unique semisimple complex such that $p_2^*L''=p_1^*L$, and $K''$ a unique semisimple complex such that ${p'_1}^*\Phi_{\mb{W}}(L)={p'_2}^*K''$. \\
\quad By Proposition \ref{FD}, we have
\eqn
\widetilde{p'_2}^*\Phi_{E_3}((\iota_3)_!L'')&=&\Phi_{E_2}(\widetilde{p_2}^*(\iota_3)_!L'')=\Phi_{E_2}((\iota_2)_!p_2^*L'')=\Phi_{E_2}((\iota_2)_!p_1^*L)\\
&=&{}^t\iota_2^*\Phi'(\widetilde{p_1}^*u_1^*L)[d_2] \\
&=&{}^t\iota_2^*({}^tp_1)_!(u'_1)^*\Phi_{\mb{W}}(L)[d_2+d_1] \\
&=&(\iota'_2)_!\widetilde{p'_1}^*\Phi_{\mb{W}}(L)[d_1+d_2]=(\iota'_2)_!(p'_1)^*\Phi_{\mb{W}}(L)[d_1+d_2] \\
&=&(\iota'_2)_!(p'_2)^*K''[d_1+d_2]=\widetilde{p'_2}^*(\iota'_3)_!K''[d_1+d_2]
\eneqn
where
\[
d_1=\rank(E_1)-\rank(\tpE{\Omega}), \quad  
d_2=\rank(E_2)-\rank(\tpE{\Omega}).
\]
Hence $\Phi_{E_3}((\iota_3)_!L)=(\iota'_3)_!K''[d_1+d_2]$. Then
\[
\Phi_{\mb{V}}((p_3)_!L'')=(u'_3)_!\Phi_{E_3}((\iota_3)_!L'')=(u'_3)_!(\iota'_3)_!K''[d_1+d_2]=(p'_3)_!K''[d_1+d_2].
\]
We have
\[
\Phi_{\mb{V}} \circ F_i^{\Omega}(L)=F_i^{\Omega'} \circ \Phi_{\mb{W}}(L)[d]
\]
where
\[
d=d_1+d_2+\sum_{i \stackrel{\Omega}{\longrightarrow} \xi}\dim{W_\xi}-\sum_{i \stackrel{\Omega'}{\longrightarrow} \eta}\dim{W_\eta}.
\]
Now we suppose $\Omega \backslash \Omega'=\{h,\theta(h)\}$ and put $\vo(h)=k,\vi(h)=\ell$. When $k=i$, we have $\sum_{i \stackrel{\Omega}{\longrightarrow} \xi}\dim{W_\xi}-\sum_{i \stackrel{\Omega'}{\longrightarrow} \eta}\dim{W_\eta}=\dim{\mb{W}_\ell}$. If $\ell \neq \theta(i)$, we have $d_2=0$ and $d_1=\dim{\mb{W}_i}\dim{\mb{W}_\ell}-\dim{\mb{V}_i}\dim{\mb{V}_\ell}=-\dim{\mb{W}_\ell}$. If $\ell=\theta(i)$, we have $d_2=0$ and $d_1=\dim\Alt(\mb{W}_i,\mb{W}_{\theta(i)})-\dim\Alt(\mb{V}_i,\mb{V}_{\theta(i)})=-\dim\mb{W}_{\theta(i)}$. Thus we obtain $\Phi_{\mb{V}} \circ F_i^{\Omega}(L)=F_i^{\Omega'} \circ \Phi_{\mb{W}}(L)$. When $k=\theta(i)$, we can prove the claim by the same way. \\
(2) \ We may suppose $\Omega \backslash \Omega'=\{h,\theta(h)\}$ and put $\vo(h)=k,\vi(h)=\ell$. \\
\quad We considet the following diagram:
\[
\xymatrix{
\tEO{W}{\Omega}\ar[d] & \tEE{W,V}{\Omega}\ar[l]_{p}\ar[r]^{\iota}\ar[d] & \tEO{V}{\Omega}\ar[d] \\
\tEO{W}{\Omega \cap \Omega'} & \tEE{W,V}{\Omega \cap \Omega'}\ar[l]\ar[r] & \tEO{V}{\Omega \cap \Omega'} \\
\tEO{W}{\Omega'}\ar[u] & \tEE{W,V}{\Omega'}\ar[l]_{p'}\ar[r]^{\iota'}\ar[u] & \tEO{V}{\Omega'}\ar[u] \\
}
\]
If $k,\ell \neq i,\theta(i)$, the above four diagrams are cartesian. Then the commutativity is clear. \\
\quad When $k=i$, we consider the two fiber products by the following:
\[
\xymatrix{
\tEO{W}{\Omega}\ar[dd] & & \tEE{W,V}{\Omega}\ar[ll]_{p}\ar[rr]^{\iota}\ar[dd]\ar@{}[ddrr]|\Box\ar[dl]_{q_2} & & \tEO{V}{\Omega}\ar[dd] \\
 & E\ar[dr]\ar[ul]_{q_1}\ar@{}[dl]|\Box & & & \\
\tEO{W}{\Omega \cap \Omega'} & & \tEE{W,V}{\Omega \cap \Omega'}\ar[ll]\ar[rr] & & \tEO{V}{\Omega \cap \Omega'} \\
 & & & E'\ar[ul]\ar[dr]_{r_1}\ar@{}[ur]|\Box & \\
\tEO{W}{\Omega'}\ar[uu]\ar@{}[uurr]|\Box & & \tEE{W,V}{\Omega'}\ar[ll]_{p'}\ar[rr]^{\iota'}\ar[uu]\ar[ur]_{r_2} & & \tEO{V}{\Omega'}\ar[uu] \\
}
\]
where
\[
E\colon =\tEO{W}{\Omega} \times_{\tEO{W}{\Omega \cap \Omega'}} \tEE{W,V}{\Omega \cap \Omega'}, \quad 
E'\colon =\tEO{V}{\Omega'} \times_{\tEO{V}{\Omega \cap \Omega'}} \tEE{W,V}{\Omega \cap \Omega'}.
\]
We can regard $E$ and $E'$ as the dual vector bundle of $\tEE{W,V}{\Omega'}$ and $\tEE{W,V}{\Omega}$ on $\tEE{W,V}{\Omega \cap \Omega'}$ respectively. We can regard $r_2$ as the transpose of $q_2$. We denote by $\Phi$ and $\Phi'$ the Fourier transforms 
\begin{eqnarray*}
\Phi&:&\ms{D}_{(\bb{C}^\times)^{\Omega,\theta}\rm{-mono}}(E) \to \ms{D}_{(\bb{C}^\times)^{\Omega,\theta}\rm{-mono}}(\tEE{W,V}{\Omega'}),\\
\Phi'&:&\ms{D}_{(\bb{C}^\times)^{\Omega,\theta}\rm{-mono}}(\tEE{W,V}{\Omega}) \to \ms{D}_{(\bb{C}^\times)^{\Omega,\theta}\rm{-mono}}(E').
\end{eqnarray*}
\quad Then, for $K \in \tmQ{V}{\Omega}$ we have
\eqn
\Phi_{\mb{W}}(p_!\iota^*K)=p'_!\Phi((q_2)_!\iota^*K)=p'_!r_2^*\Phi'(\iota^*K)[d]=p'_!r_2^*r_1^*\Phi_{\mb{V}}(K)[d]=p'_!(\iota')^*\Phi_{\mb{V}}(K)[d],
\eneqn
where $d=\rank(E)-\rank(\tEE{W,V}{\Omega})$. If $\ell \neq \theta(i)$, we have $\rank(E)=\dim{\mb{W}_i}\dim{\mb{W}_\ell}$ and $\rank(\tEE{W,V}{\Omega})=\dim{\mb{V}_i}\dim{\mb{V}_\ell}$. Since $\mb{V}_\ell=\mb{W}_\ell$, we have $d=-\dim\mb{W}_\ell$. If $\ell=\theta(i)$, we have $\rank(E)=\dim{\Alt(\mb{W}_i,\mb{W}_{\theta(i)})}$ and $\rank(\tEE{W,V}{\Omega})=\dim{\Alt(\mb{V}_i,\mb{V}_{\theta(i)})}$. Then $d=-\dim\mb{W}_{\theta(i)}$. Since $\Omega \backslash \Omega'=\{i \to \ell,\theta(\ell) \to \theta(i)\}$, we have $d_{E^\Omega_i}-\dim{\mb{W}_\ell}=d_{E^{\Omega'}_i}$. Thus $\Phi_{\mb{W}} \circ E^\Omega_i(K)=E_i^{\Omega'} \circ \Phi_{\mb{V}}$.\\
\quad When $k=\theta(i)$, we obtain the following diagram:
\[
\xymatrix{
\tEO{W}{\Omega}\ar[dd]\ar@{}[ddrr]|\Box & & \tEE{W,V}{\Omega}\ar[ll]_{p}\ar[rr]^{\iota}\ar[dd]\ar[dr]^{\iota_2} & & \tEO{V}{\Omega}\ar[dd] \\
 & & & F\ar[dl]\ar[ur]\ar@{}[dr]|\Box& \\
\tEO{W}{\Omega \cap \Omega'} & & \tEE{W,V}{\Omega \cap \Omega'}\ar[ll]\ar[rr]\ar@{}[ddrr]|\Box & & \tEO{V}{\Omega \cap \Omega'} \\
 & F'\ar[ur]\ar[dl]\ar@{}[ul]|\Box& & & \\
\tEO{W}{\Omega'}\ar[uu] & & \tEE{W,V}{\Omega'}\ar[ll]_{p'}\ar[rr]^{\iota'}\ar[uu]\ar[ul]_{p'_2} & & \tEO{V}{\Omega'}\ar[uu] \\
}
\]
Here \[
F\colon =\tEO{V}{\Omega} \times_{\tEO{V}{\Omega \cap \Omega'}} \tEE{W,V}{\Omega \cap \Omega'}, \quad 
F'\colon =\tEO{W}{\Omega'} \times_{\tEO{W}{\Omega \cap \Omega'}} \tEE{W,V}{\Omega \cap \Omega'}.
\]
We regard $p'_2$ as the transpose of $\iota_2$. Hence we can prove the claim by the similar way. \\
(3) The claim follows from Proposition \ref{FD}(2) and the commutativity of $F_i$ and $\Phi_{\mb{V}}^{\Omega\Omega'}$.
\end{proof}
Similarly, we can prove the commutativity of $F_i^{(a)}$'s and the Fourier transforms. We omit the proof. 
\begin{prop}\label{prop-comdiv}
Let $\mb{W}$ and $\mb{V}$ be $\theta$-symmetric $I$-graded vector spaces such that $\wt(\mb{V})=\wt(\mb{W})+a(\alpha_i+\alpha_{\theta(i)})$. Let ${F_i^{(a)}}^\Omega$ and ${F_i^{(a)}}^{\Omega'}$ be the $a$-th divided powers with respect to $\theta$-orientations $\Omega$ and $\Omega'$, respectively. For a $\tG{W}$-equivariant semisimple obejct $L \in \tmQ{W}{\Omega}$, we have $\Phi_{\mb{V}}^{\Omega\Omega'} \circ {F_i^{(a)}}^{\Omega}(L) \cong {F_i^{(a)}}^{\Omega'} \circ \Phi_{\mb{W}}^{\Omega\Omega'}(L)$.

\end{prop}

\section{A Geometric Constuction of Symmetric Crystals}
\subsection{Grothendieck group}
For a $\theta$-orientation $\Omega$ and a $\theta$-symmetric and $I$-graded vector space $\mb{V}$, we define $\tK{V}{\Omega}$ as the Grothendieck group of $\tmQ{V}{\Omega}$. Namely $\tK{V}{\Omega}$ is generated by $(L)$ for $L \in \tmQ{V}{\Omega}$ with the relation $(L)=(L')+(L'')$ when $L \cong L' \oplus L''$. This is a $\bb{Z}[v,v^{-1}]$-module by $v(L)=(L[1])$ and $v^{-1}(L)=(L[-1])$ for $L \in \tmQ{V}{\Omega}$. Hence, $\tK{V}{\Omega}$ is a free $\bb{Z}[v,v^{-1}]$-module with a basis $\{(L) \ | \ L \in \tmP{V}{\Omega}\}$. For another $\theta$-symmetric and $I$-graded vector space $\mb{V}'$ such that $\ud{\mb{V}}=\ud{\mb{V}'}$, we have $\tK{V}{\Omega} \cong \tK{V'}{\Omega}$. We define 
\[
\tKK{\Omega}\colon =\bigoplus_{\mb{V}}\tK{V}{\Omega}
\]
where $\mb{V}$ runs over the isomorphism classes of $\theta$-symmetric $I$-graded vector spaces. For two $\theta$-orientations $\Omega$ and $\Omega'$, the Fourier transform induces an equivalence $\tmQ{V}{\Omega} \to \tmQ{V}{\Omega'}$ and the isomorphism $\tK{V}{\Omega} \isoto \tK{V}{\Omega'}$. Therefore $\tKK{\Omega} \cong \tKK{\Omega'}$. \\
\quad We set ${}^\theta\!K=\tKK{\Omega},{}^\theta\!\!\ms{P}_{\mb{V}}=\tmP{V}{\Omega}$. By Lemma \ref{FD-lem}, they are well-defined.

\subsection{Actions of $E_i$ and $F_i$}
The functors $E_i$ and $F_i^{(a)}$ induce the action on $\tKK{\Omega}$.Since $E_i$ and $F_i$ commute with the Fourier transforms, they also act on ${}^\theta\!K$. The submodule ${}^\theta\!K'\colon =\sum_{(\mb{i},\mb{a})}\bb{Z}[v,v^{-1}](\tLL{i}{a}{\Omega}) \subset {}^\theta\!K$ is stable by $E_i$ and $F_i$ by Proposition \ref{prop-ef}. We define
\[
T_i|_{\tK{V}{\Omega}}=v^{-(\al_i,\wt{\mb{V}})}\id_{\tK{V}{\Omega}}.
\]
\begin{prop}\label{prop:action}
The operators $E_i,F_i$ and $T_i \ (i \in I)$ regarded as operators on ${}^\theta\!K'$ satisfy
\eqn
E_iF_j-v^{-(\alpha_i,\alpha_j)}F_jE_i=\delta_{ij}+\delta_{\theta(i),j}T_i
\eneqn
and 
\eqn
T_iE_jT_i^{-1}=v^{(\alpha_i+\alpha_{\theta(i)},\alpha_j)}E_j, \quad 
T_iF_jT_i^{-1}=v^{(\alpha_i+\alpha_{\theta(i)},-\alpha_j)}F_j.
\eneqn
\end{prop}
\begin{proof}
We take $\theta$-symmetric $I$-graded vector spaces $\mb{W},\mb{V},\mb{U}$ and $\mb{X}$ such that $\wt(\mb{V})=\wt(\mb{W})+\alpha_j+\alpha_{\theta(j)}$, $\wt(\mb{U})=\wt(\mb{W})+\alpha_i+\alpha_{\theta(i)}$ and $\wt(\mb{X})=\wt(\mb{W})+\alpha_j+\alpha_{\theta(j)}+\alpha_{i}+\alpha_{\theta(i)}$. We consider the following diagram:
\[
\xymatrix{
\ms{D}(\tEO{W}{\Omega})\ar[r]^{F_j}\ar[d]_{E_i} & \ms{D}(\tEO{V}{\Omega})\ar[d]^{E_i} \\
\ms{D}(\tEO{U}{\Omega})\ar[r]_{F_j} & \ms{D}(\tEO{X}{\Omega})
}
\]
\quad First, we have
\eqn
E_iF_j\tLL{i}{a}{\Omega}&=&\delta_{ij}{}^\theta\!{L}_{(i,\mb{i},\theta(i)),(0,\mb{a},0);\Omega}[c_\mb{a}]\oplus\delta_{\theta(i),j}{}^\theta\!L_{(\theta(i),\mb{i},i),(0,\mb{a},0);\Omega}[c_{\mb{a}_\theta}]\oplus\bigoplus_{\mb{a}'}{}^\theta\!L_{(j,\mb{i},\theta(j)),(1,\mb{a}',1);\Omega}[c_{\mb{a'}}],
\eneqn
where
\eqn
c_\mb{a}&=&\dim{\mb{W}}_i+\sum_{i \to \eta}\dim{\mb{W}_\eta}-\dim{\mb{X}}_i+\sum_{i \to \xi}\dim{\mb{X}_\xi}-2M_1((i,\mb{i},\theta(i)),(0,\mb{a},0)), \\
c_{\mb{a}_\theta}&=&\dim{\mb{W}_{\theta(i)}}+\sum_{\theta(i) \to \eta}\dim{\mb{W}_\eta}-\dim{\mb{X}_i}+\sum_{i \to \xi}\dim{\mb{X}_\xi}-2M_{2m+1}((\theta(i),\mb{i},i),(0,\mb{a},0)), \\
c_{\mb{a'}}&=&\dim{\mb{W}_j}+\sum_{j \to \eta}\dim{\mb{W}_j}-\dim{\mb{X}_i}+\sum_{i \to \xi}\dim{\mb{X}_\xi}-2M_{k+1}((j,\mb{i},\theta(j)),(1,\mb{a}',1)).
\eneqn
Here $\mb{a}'$ runs over the sequences $\mb{a}^{(k)} \ (1 \le k \le m,i_k=i,\theta(i))$. \\
 \quad If $i=j$, we have $c_{\mb{a}}=0$ by $\mb{W}=\mb{X}$ and 
\[
M_1((i,\mb{i},\theta(i)),(0,\mb{a},0))=\sum_{i \to i_\ell}a_\ell=\sum_{i \to \eta}\dim{\mb{W}_\eta}.
\]
\quad If $\theta(i)=j$, we have
\[
c_{\mb{a}_\theta}=\sum_{\eta \to i}\dim{\mb{W}_\eta}+\sum_{i \to \eta}\dim{\mb{W}_\eta}-2\dim{\mb{W}_i}=-(\alpha_i,\wt(\mb{V}))
\]
by $\mb{W}=\mb{X}$, 
\[
M_{2m+1}((\theta(i),\mb{i},i),(0,\mb{a},0))=\sum_{i_\ell=i}a_\ell=\dim{\mb{W}_i}
\]
and $\sum_{\theta(i) \to \eta}\dim{\mb{W}_\eta}=\sum_{\eta \to i}\dim{\mb{W}_\eta}$. \\
\quad On the other hand, we have
\[
F_jE_i\tLL{i}{a}{\Omega}=\bigoplus_{\mb{a}'}{}^\theta\!L_{(j,\mb{i},\theta(j)),(1,\mb{a}',1);\Omega}[d_{\mb{a'}}],
\]
where
\eqn
d_{\mb{a}'}=-\dim{\mb{U}_i}+\sum_{i \to \xi}\dim{\mb{U}_\xi}+\dim{\mb{U}_j}+\sum_{j \to \eta}\dim{\mb{U}_\eta}-2M_k(\mb{i},\mb{a}').
\eneqn
and $\mb{a}'$ runs over the sequences $\mb{a}^{(k)} \ (1 \le k \le m,i_k=i,\theta(i))$. \\
\quad We have 
\[
M_{k+1}((j,\mb{i},\theta(j)),(1,\mb{a}',1))-M_k(\mb{i},\mb{a}')=\left\{
\begin{array}{ll}
1+\sharp\{i \stackrel{\Omega}{\longrightarrow} \theta(i)\} & (j=i) \\
0 & (j=\theta(i)) \\
\sharp\{i \stackrel{\Omega}{\longrightarrow} \theta(j)\} & (j \neq i,\theta(i))
\end{array}
\right.
\]
and 
\eqn
&{}&\left(\dim{\mb{W}_j}+\sum_{j \to \eta}\dim{\mb{W}_j}-\dim{\mb{X}_i}+\sum_{i \to \xi}\dim{\mb{X}_\xi}\right)\\
&{}& \quad \quad -\left(-\dim{\mb{U}_i}+\sum_{i \to \xi}\dim{\mb{U}_\xi}+\dim{\mb{U}_j}+\sum_{j \to \eta}\dim{\mb{U}_\eta}\right) \\
&=&\left\{
\begin{array}{ll}
2\sharp\{i \stackrel{\Omega}{\longrightarrow} \theta(i)\} & (j=i)\\
\sharp\{i \stackrel{\Omega}{\longrightarrow} \theta(i)\}+\sharp\{\theta(i) \stackrel{\Omega}{\longrightarrow} i\} & (j=\theta(i)) \\
\sharp\{i \stackrel{\Omega}{\longrightarrow} j\}+\sharp\{j \stackrel{\Omega}{\longrightarrow} i\}+2\sharp\{i \stackrel{\Omega}{\longrightarrow} \theta(j)\} & (j \neq i,\theta(i))
\end{array}
\right..
\eneqn
Thus $c_{\mb{a}'}-d_{\mb{a}'}=-(\alpha_i,\alpha_j)$. We conclude 
\[
E_iF_j(\tLL{i}{a}{\Omega})-v^{-(\alpha_i,\alpha_j)}F_jE_i(\tLL{i}{a}{\Omega})=\delta_{ij}(\tLL{i}{a}{\Omega})+\delta_{\theta(i),j}T_i(\tLL{i}{a}{\Omega}).
\]
\quad The relations $T_iE_jT_i^{-1}=v^{(\alpha_i+\alpha_{\theta(i)},\alpha_j)}E_j$ and $T_iF_jT_i^{-1}=v^{(\alpha_i+\alpha_{\theta(i)},-\alpha_j)}F_j$ are obvious.
\end{proof}

\subsection{Key estimates of coefficients}
Let $\Omega$ be a $\theta$-orientation and suppose that a vertex $i$ is a sink. For a $\theta$-symmetric $I$-graded vector space $\mb{V}$  and $r \in \bb{Z}_{\ge 0}$, we define
\[
\tEO{V}{\Omega,r}\colon =\left\{
x \in \tEO{V}{\Omega}\left|\dim\Coker\left(
\bigoplus_{h \in \Omega;\vi(h)=i}\mb{V}_{\vo(h)} \to \mb{V}_i
\right)=r
\right.
\right\}.
\]
Then we have $\tEO{V}{\Omega}=\sqcup_{r \ge 0}\tEO{V}{\Omega,r}$, and $\tEO{V}{\Omega,\ge{r}}\colon =\sqcup_{r' \ge r}\tEO{V}{\Omega,r'}$ is a closed subset of $\tEO{V}{\Omega}$. 
\begin{dfn}
For $L \in {}^\theta\!\!\ms{P}_{\mb{V}}$ and $i \in I$, choose a $\theta$-orientation $\Omega$ such that $i$ is a sink with respect to $\Omega$, and regard $L$ as an element of $\tmP{V}{\Omega}$. We define $\varepsilon_i(L)$ as the largest integer $r$ satisfying $\Supp(L) \subset \tEO{V}{\Omega,\ge{r}}$. This does not depend on the choice of $\Omega$. 
\end{dfn}
Note that $0 \le \varepsilon_i(L) \le \dim{V_i}$. \\
\quad We shall prove the following key estimates with respect to $F_i(L)$ and $E_i(L)$.
\begin{thm}\label{est}
Assume that $\theta$-symmetric and $I$-graded vector spaces $\mb{V}$ and $\mb{W}$ satisfy $\wt{\mb{V}}=\wt{\mb{W}}+\alpha_i+\alpha_{\theta(i)}$. Fix a $\theta$-orientation $\Omega$ such that the vertex $i$ is a sink.
\benu[{\rm(1)}]
\item For $L \in \tmP{W}{\Omega}$, there exists a unique simple perverse sheaf $L_0 \in \tmP{V}{\Omega}$ such that $\varepsilon_i(L_0)=\varepsilon_i(L)+1$ and 
\[
F_i(L)=[\varepsilon_i(L)+1]_v(L_0)+\sum_{L' \in \tmP{V}{\Omega}\colon \varepsilon_i(L')>\varepsilon_i(L)+1}a_{L'}(L')
\]
for $a_{L'} \in v^{2-\varepsilon_i(L')}\bb{Z}[v]$. \\
\quad We define the map $\tFi\colon {}^\theta\!\!\ms{P}_{\mb{W}}\cong\tmP{W}{\Omega} \to \tmP{V}{\Omega}\cong{}^\theta\!\!\ms{P}_{\mb{V}}$ by $\tFi(L)=L_0$. It does not depend on the choice of $\Omega$. \\
\item Let $K \in \tmP{V}{\Omega}$. If $\varepsilon_i(K)>0$, there exists a unique simple perverse sheaf $K_0 \in \tmP{W}{\Omega}$ such that $\varepsilon_i(K_0)=\varepsilon_i(K)-1$ and 
\[
E_i(K)=v^{1-\varepsilon_i(K)}(K_0)+\sum_{K' \in \tmP{W}{\Omega}\colon \varepsilon_i(K')>\varepsilon_i(K)-1}b_{K'}(K')
\]
for $b_{K'} \in v^{-\varepsilon_i(K')+1}\bb{Z}[v]$. Here we regard $K_0=0$ if $\varepsilon_i(K)=0$. \\
\quad We define the map $\tEi\colon {}^\theta\!\!\ms{P}_{\mb{V}}\cong\tmP{V}{\Omega} \to \tmP{W}{\Omega} \sqcup \{0\} \cong {}^\theta\!\!\ms{P}_{\mb{W}}\sqcup \{0\}$ by $\tEi(K)=K_0$ if $\varepsilon_i(K)>0$ and $\tEi(K)=0$ if $\varepsilon_i(K)=0$. It does not depend on the choise of $\Omega$.
\eenu
\end{thm}
\begin{proof}
(1) \ We consider the diagram 
\[
\xymatrix{
\tEO{W}{\Omega} & \tpE{\Omega}\ar[r]_{p_2}\ar[l]^(.3){p_1} & \tppE{\Omega}\ar[r]_{p_3} & \tEO{V}{\Omega}.
}
\]
Since $i$ is a sink, we have $p_1^{-1}(\tEO{W}{\Omega,r})=p_2^{-1}p_3^{-1}(\EO{V}{\Omega,r+1})$  for any integer $r$. Especially, for $L \in \tmP{W}{\Omega}$, $p_3p_2(p_1^{-1}\Supp{L}) \subset \tEO{V}{\Omega,\ge\varepsilon_i(L)+1}$. For $r$, set $\tppE{\Omega,r}=p_3^{-1}(\tEO{V}{\Omega,r})$. Then $p_2^{-1}(\tppE{\Omega,r})=p_1^{-1}(\tEO{W}{\Omega,r-1})$. We set $\tppE{\Omega,\le r}\colon =\cup_{r'\le r}\tppE{\Omega,r'}$. Then $\tppE{\Omega,\le r}$ is an open subset of $\tppE{\Omega}$. If $p_3(x,V)=x \in \tEO{V}{\Omega,r}$, $V_i$ is a one-codimensional subspace of $\mb{V}_i$ which contains the $(\dim\mb{V}_i-r)$-dimensinal subspace $\sum_{\vi(h)=i}\Ig{x_h}$ of $\mb{V}_i$. Therefore $\tppE{\Omega,r} \to \tEO{V}{\Omega,r}$ is a $\mb{P}^{r-1}$-bundle. For $L \in \tmP{W}{\Omega}$, there is a unique simple perverse sheaf $L'' \in \ms{D}(\tppE{\Omega})$ such that $p_1^*L[d_{p_1}-d_{p_2}]=p_2^*L''$ and $(p_3)_*L''=F_i(L)$. For $x \in \tEO{V}{\Omega,\varepsilon_i(L)+1}$, the action of the stabilizer ${}^\theta\!G_{\mb{V},x} \subset \tG{V}$ of $x$ on $p_3^{-1}(x)$ is transitive. Since $L''$ is $\tG{V}$-equivariant, $L''$ is a constant sheaf on any fibers of $p_3$ over $\tEO{V}{\Omega,\varepsilon_i(L)+1}$. \\
\quad We restrict $L''$ to the open subset $\tppE{\Omega,\le \varepsilon_i(L)+1}$. There exists a nuique simple perverse sheaf $J_0$ on $\tEO{V}{\Omega,\le\varepsilon_i(L)+1}$ such that $L''|_{\tppE{\Omega,\le \varepsilon_i(L)+1}}=p_3^*J_0[\varepsilon_i(L)]$. Hence $(p_3)_*L''|_{\tppE{\Omega,\le \varepsilon_i(L)+1}}=(p_3)_*p_3^*J_0[\varepsilon_i(L)]=[\varepsilon_i(L)+1]_vJ_0$. Let $L_0$ be the minimal extension of $J_0$. Then $L_0$ is a simple perverse sheaf on $\tEO{V}{\Omega}$. Since $F_iL$ is semisimple, we have 
\[
F_i(L)=[\varepsilon_i(L)+1]_v(L_0)+\sum{a_{L'}(L')},
\]
where $L' \in \tmP{V}{\Omega}$ satisfies $\Supp(L') \subset \tEO{V}{\Omega,>\varepsilon_i(L)+1}$, or $\varepsilon_i(L')>\varepsilon_i(L)$. \\
\quad To prove $a_{L'} \in v^{2-\varepsilon_i(L')}\bb{Z}[v]$, we restrict $\rHom((p_3)_*L'',L')$ to the open subset $\tEO{V}{\Omega,\le \varepsilon_i(L')}$. Write $F_iL=\oplus_{J \in \tmP{V}{\Omega}}J \otimes M_J$, where $M_{J} \in \ms{D}(\pt)$ is the multiplicity space of $J$ in the expansion of $F_iL$. Then
\begin{eqnarray*}
\rHom((p_3)_*L'',L')|_{\tEO{V}{\Omega,\le \varepsilon_i(L')}}&=&\bigoplus_{J}\rHom(J,L')|_{\tEO{V}{\Omega,\le \varepsilon_i(L')}} \otimes M_{J}^* \\
&\supset& \rHom(L',L')|_{\tEO{V}{\Omega,\le \varepsilon_i(L')}} \otimes M_{L'}^*,
\end{eqnarray*}
\quad On the other hand, since $p_3$ is a $\mb{P}^{\varepsilon_i(L')-1}$-bundle on $\tEO{V}{\Omega,\varepsilon_i(L')}$ and $\Supp(L') \cap \tEO{V}{\Omega, \le \varepsilon_i(L')} \subset \tEO{V}{\Omega,\varepsilon_i(L')}$, we have
\begin{eqnarray*}
&{}&\rHom((p_3)_*L'',L')|_{\tEO{V}{\Omega,\le \varepsilon_i(L')}} \\
&=&(p_3)_*\rHom(L'',p_3^!L')|_{p_3^{-1}(\tEO{V}{\Omega,\le \varepsilon_i(L')})} \\
&=&(p_3)_*\rHom(L'',p_3^*L'[\varepsilon_i(L')-1])|_{p_3^{-1}(\tEO{V}{\Omega,\le \varepsilon_i(L')})}[\varepsilon_i(L')-1].
\end{eqnarray*}
Since $p_3^*L'[\varepsilon_i(L')-1]|_{p_3^{-1}(\tEO{V}{\Omega,\le \varepsilon_i(L')})}$ is a perverse sheaf, we have 
\[
\rHom(L'',p^*L'[\varepsilon_i(L')-1])|_{p_3^{-1}(\tEO{V}{\Omega,\le \varepsilon_i(L')})} \in \ms{D}^{\ge 0}
\]
by Lemma \ref{lem:hom}. Moreover since $\Supp(L'') \supsetneq \Supp(p_3^*L')$, we have 
\[
H^0(\rHom(L'',p_3^*L'[\varepsilon_i(L')-1])|_{p_3^{-1}(\tEO{V}{\Omega,\le \varepsilon_i(L')})})=0. 
\]
Therefore $\rHom(L'',p^*L'[\varepsilon_i(L')-1])|_{p_3^{-1}(\tEO{V}{\Omega,\le \varepsilon_i(L')})} \in \ms{D}^{>0}(\pt)$ and its direct image of $p_3$ is contained in $\ms{D}^{>0}$. Thus we obtain $\rHom((p_3)_*L'',L')|_{\tEO{V}{\le \varepsilon_i(L')}} \in \ms{D}^{>1-\varepsilon_i(L')}$. \\
\quad Since $H^0(\rHom(L',L')) \neq 0$, we conclude $M_{L'}^* \in \ms{D}^{>1-\varepsilon_i(L')}(\pt)$. Hence $a_{L'} \in v^{2-\varepsilon_i(L')}\bb{Z}[v]$. \\
(2) \ Recall the following diagram:
\[
\xymatrix{
\EO{T}{\Omega} \times \tEO{W}{\Omega} & \tEE{W,V}{\Omega}\ar[l]_(.45){p}\ar@{^{(}->}[r]^(.6){\iota} & \tEO{V}{\Omega}
}.
\]
Since $i$ is a sink, for a fixed $x_\mb{W} \in \tEO{W}{\Omega}$, $x \in \tEO{V}{\Omega}$ is uniquely determined by the condition that $U$ is $x$-stable and $x$ induces $x_\mb{W}$ on $U/U^\bot \cong \mb{W}$. Therefore we have $\tEO{W}{\Omega} \cong \tEE{W,V}{\Omega}$. We have a section $s$ of $p_1\colon \tpE{\Omega} \to \tEO{W}{\Omega}$ by $x_\mb{W} \mapsto (x,U,\varphi_\mb{W})$ where $\varphi_{\mb{W}}\colon U/U^\bot \cong \mb{W}$ is a given isomorphism of $\theta$-symmetric $I$-graded vector spaces. We consider the following diagram:
\[
\xymatrix{
 \tpE{\Omega}\ar[d]^{p_1}\ar[r]^{q\colon =p_3 \circ p_2} & \tEO{V}{\Omega} \\
 \tEO{W}{\Omega}\ar[r]^(.4){\sim}\ar@/^1pc/[u]^s & \tEE{W,V}{\Omega}\ar[u]_{\iota}
}
\]
For $K \in \tmP{V}{\Omega}$, we have $E_iK=s^*q^*K[-\dim{\mb{W}_i}]$. Assume that $\varepsilon_i(K)>0$. Since $\Supp(K) \subset \tEO{V}{\Omega, \ge \varepsilon_i(K)}$, $K|_{\tEO{V}{\Omega,\varepsilon_i(K)}}$ is a simple perverse sheaf. Since $q$ is smooth on $\tEO{V}{\Omega,\varepsilon_i(K)}$, the restriction $q^*K[d_q]|_{q^{-1}(\tEO{V}{\Omega,\varepsilon_i(K)})}$ is a $\tG{V}$-equivariant perverse sheaf, where $d_q$ is the fiber dimension of $q$ on $\tEO{V}{\Omega,\varepsilon_i(K)}$. Note that $p_1$ is an affine bundle on $\tEO{W}{\Omega,\varepsilon_i(K)-1}$. If $x \in \tEE{W,V}{\Omega}$ induces $x_\mb{W} \in \tEO{W}{\Omega,\varepsilon_i(K)-1}$, the stabilizer ${}^\theta\!G_{\mb{V}}$ acts transitively on the fiber of $p_1$ at $x_{\mb{W}}$. Thus $q^*K[d_q]|_{q^{-1}(\tEO{V}{\Omega,\varepsilon_i(K)})}$ is constant on any fibers of $p_1$. Hence $s^*q^*K[d_q-d_{p_1}]|_{\tEO{W}{\Omega,\varepsilon_i(K)-1}}$ is a simple perverse sheaf. Here
\[
d_{p_1}-d_q=d_{p_1}-d_{p_2}-(\varepsilon_i(K)-1)=\dim\mb{W}_i+1-\varepsilon_i(K).
\]
Therefore we obtain 
\[
E_i(K)=v^{1-\varepsilon_i(K)}(K_0)+\sum_{K' \in \tmP{W}{\Omega}\colon \varepsilon_i(K')>\varepsilon_i(K)-1}b_{K'}(K'),
\]
where $K_0$ is the minimal extension of $s^*q^*K[d_q-d_{p_1}]|_{\tEO{W}{\Omega,\varepsilon_i(K)-1}}$. \\
\quad We shall prove $b_{K'} \in v^{1-\varepsilon_i(K')}\bb{Z}[v]$. \\
\quad Since $q^*K[-\dim{\mb{W}_i}]$ and $p_1^*E_iK$ are constant along the fibers of $p_1$, and $s^*q^*K[-\dim{\mb{W}_i}]=s^*p_1^*E_iK$, we obtain $q^*K[-\dim{\mb{W}_i}]=p_1^*E_iK$. We have $q^*K[-\dim{\mb{W}_i}]=\oplus_{K''}{p_1^*K'' \otimes M_{K''}}$, where $M_{K''}$ is the multiplicity space of $K''$ in $E_iK$. Since there is a unique semisimple object $L_{K''} \in \ms{D}(\tppE{\Omega})$ such that $p_1^*K''=p_2^*L_{K''}$, we have $p_2^*p_3^*K[-\dim{\mb{W}_i}]=\oplus_{K''}{p_2^*L_{K''} \otimes M_{K''}}$. We obtain $p_3^*K[-\dim{\mb{W}_i}]=\oplus_{K''}L_{K''} \otimes M_{K''}$. \\
\quad Therefore we have
\begin{eqnarray*}
&{}&\oplus_{K''}\mb{R}\!\Hom(L_{K''}|_{\tEO{V}{\Omega,\le\varepsilon_i(K')+1}},L_{K'}|_{\tEO{V}{\Omega,\le\varepsilon_i(K')+1}}) \otimes M_{K''}^* \\
&=&\mb{R}\!\Hom(p_3^*K[-\dim{\mb{W}_i}]|_{p_3^{-1}(\tEO{V}{\Omega,\le\varepsilon_i(K')+1})},L_{K'}|_{p_3^{-1}(\tEO{V}{\Omega,\le\varepsilon_i(K')+1})}) \\
&=&\mb{R}\!\Hom(K[-\dim{\mb{W}_i}]|_{\tEO{V}{\Omega,\le\varepsilon_i(K')+1}},(p_3)_*L_{K'}|_{\tEO{V}{\Omega,\le\varepsilon_i(K')+1}}) \\
&=&\mb{R}\!\Hom(K[-\dim{\mb{W}_i}]|_{\tEO{V}{\Omega,\le\varepsilon_i(K')+1}},F_i(K')|_{\tEO{V}{\Omega,\le\varepsilon_i(K')+1}}[-\dim{\mb{W}_i}]) \\
&=&\mb{R}\!\Hom(K|_{\tEO{V}{\Omega,\varepsilon_i(K')+1}},F_i(K')|_{\tEO{V}{\Omega,\le\varepsilon_i(K')+1}}).
\end{eqnarray*}
By the claim of (1), $F_iK'|_{\tEO{V}{\Omega,\le\varepsilon_i(K')}}=[\varepsilon_i(K')+1]_v\tFi{K'}|_{\tEO{V}{\Omega,\le\varepsilon_i(K')+1}} \in {}^p\!\ms{D}^{\ge -\varepsilon_i(K')}(\tEO{V}{\Omega,\le\varepsilon_i(K')+1})$. Since $\Supp(K) \subsetneq \Supp(F_iK')$, we have $\mb{R}\!\Hom(K|_{\tEO{V}{\Omega,\le\varepsilon_i(K')+1}},F_i(K')|_{\tEO{V}{\Omega,\le\varepsilon_i(K')+1}}) \in \ms{D}^{\ge 1-\varepsilon_i(K')}$. Therefore $\mb{R}\!\Hom(L_{K'},L_{K'})\otimes M_{K'}^* \in \ms{D}^{\ge1-\varepsilon_i(K')}$, which implies $M_{K'}^* \in \ms{D}^{\ge 1-\varepsilon_i(K')}$. Hence $b_{K'} \in v^{1-\varepsilon_i(K')}\bb{Z}[v]$ is proved. \\
\quad In the case $\varepsilon_i(K)=0$, we can prove similarly $b_{K'} \in v^{1-\varepsilon_i(K')}\bb{Z}[v]$.
\end{proof}

\begin{lem}\label{lem:eps} \ Suppose $\wt\mb{V} \neq 0$. For any $L \in \tmP{V}{\Omega}$, there exists $i \in I$ such that $\varepsilon_i(L)>0$.
\end{lem}
\begin{proof}
If $\mb{V} \neq \{0\}$, there exists an integer $d$, $\mb{i}=(i_1, \ldots ,i_{2m})$ and $\mb{a}$ such that $L[d]$ appears in a direct summand of $\tLL{i}{a}{\Omega}$. We may assume $a_1>0$. Then, taking $\Omega$ such that $i_1$ is a sink, we have $\Supp(L) \subset \Supp(\tLL{i}{a}{\Omega}) \subset \tEO{V}{\Omega,\ge 1}$. By the definition of $\varepsilon_i$, we have $\varepsilon_{i_1}(L) \neq 0$. \end{proof}

\begin{lem}\label{lem:cef}
For $L \in \ms{P}_{\mb{V}}$, we have $\tEi\tFi(L)=(L)$, and if $\tEi(L) \neq 0$, we have $\tFi\tEi(L)=L$. 
\end{lem}
\begin{proof}
We assume that $i$ is a sink. \\
\quad Recall the following diagram:
\[
\xymatrix{
 \tpE{\Omega}\ar[d]^{p_1}\ar[r]^{q\colon =p_3 \circ p_2} & \tEO{V}{\Omega} \\
 \tEO{W}{\Omega}\ar[r]^(.4){\sim}\ar@/^1pc/[u]^s & \tEE{W,V}{\Omega}\ar[u]_{\iota}
}
\]
\quad For $L \in \tmP{W}{\Omega}$, take simple perverse sheaf $L'' \in \ms{D}(\tppE{\Omega})$ such that $p_1^*L[\dim\mb{W}_i]=p_2^*L''$ and $(p_3)_!L''=F_iL$, then $(p_3)_!L'' \cong [\varepsilon_i(L)+1]_v\tFi{L}$ on $\tEO{V}{\Omega,\le \varepsilon_i(L)+1}$. On the other hand, since $L'' \cong p_3^*\tFi{L}[\varepsilon_i(L)]$ on $p_3^{-1}(\tEO{V}{\Omega,\le \varepsilon_i(L)+1})$, we have $q^*\tFi{L} \cong p_2^*L''[-\varepsilon_i(L)]=p_1^*L[\dim\mb{W}_i-\varepsilon_i(L)]$ on $p_1^{-1}(\tEO{W}{\Omega,\le \varepsilon_i(L)})$. Then we have $s^*q^*\tFi{L}=L[\dim\mb{W}_i-\varepsilon_i(L)]$ on $\tEO{W}{\Omega,\le \varepsilon_i(L)+1}$. Note that $E_i\tFi{L}=s^*q^*\tFi{L}[-\dim\mb{W}_i]$. We obtain $E_i(\tFi{L})=L[-\varepsilon_i(L)]$ on $\tEO{W}{\Omega,\le \varepsilon_i(L)}$. Hence $\tEi\tFi(L)=(L)$. \\
\quad Conversely, take $K \in \tmP{V}{\Omega}$ such that $\varepsilon_i(K)>0$. By the similar argument in the proof of Theorem \ref{est}(2), we have $p_1^*E_iK=q^*K[-\dim\mb{W}_i]$. Hence we obtain $p_1^*\tEi{K}[\dim\mb{W}_i]=q^*K[\varepsilon_i(K)-1]$ on $p_1^{-1}(\tEO{W}{\Omega,\varepsilon_i(K)-1})$. Since $p_3^*K[\varepsilon_i(K)-1]$ is a simple perverse sheaf on $p_3^{-1}(\tEO{V}{\Omega,\varepsilon_i(K)})$, we have $F_i\tEi{K}=(p_3)_*p_3^*K[\varepsilon_i(K)-1]=[\varepsilon_i(K)]_vK$ on $\tEO{V}{\Omega,\le\varepsilon_i(K)}$. Then we have $\tFi\tEi(K)=(K)$. 
\end{proof}

\subsection{Verdier duality functor}

The Verdier duality functor $D\colon \ms{D}(\tEO{V}{\Omega}) \to \ms{D}(\tEO{V}{\Omega})$ satisfies $D(L[d])=D(L)[-d]$ for $L \in \ms{D}(\tEO{V}{\Omega})$, $d \in \bb{Z}$. Then $D$ induces the involution $v^{\pm{1}} \mapsto v^{\mp{1}}$.
\begin{prop}\label{prop-V} \ 
\benu[(i)]
\item $D(\tLL{i}{a}{\Omega})=\tLL{i}{a}{\Omega}[2\dim{}^\theta\tFl{i}{a}{\Omega}]$.
\item For any $L \in \tmQ{V}{\Omega}$, we have $D(F_iL)=F_iD(L)$.
\item For any $L \in \tmP{V}{\Omega}$, we have $D(L) \cong L$.
\eenu
\end{prop}
\begin{proof}
\quad (i) and (ii) follow from the general property of the Verdier duality functor (see Lemma \ref{lem-VD}).\\
\quad To prove (iii), we use the induction on $\ud{\mb{V}}$. \\
\quad When $\ud{\mb{V}}=0$, the claim is clear by $\tmP{V}{\Omega}=\{\mb{1}_{\pt}\}$ and $D(\mb{1}_{\pt})=\mb{1}_{\pt}$. \\
\quad Suppose $\ud{\mb{V}} \neq 0$. By Lemma \ref{lem:eps}, there exists $i$ such that $\varepsilon_i(L)>0$. We shall prove $D(L)=L$ by the descending induction on $\varepsilon_i(L)$. By Theorem \ref{est} and Lemma \ref{lem:cef}, we have 
\[
F_i(\tEi{L})=[\varepsilon_i(L)]_v(L)+\sum_{L' \in \tmP{V}{\Omega}\colon \varepsilon_i(L')>\varepsilon_i(L)}a_{L'}(L').
\]
By the induction hypothesis on $\wt{\mb{V}}$, $D(\tEi{L})=\tEi{L}$. Hence the lefthand side is $D$-invariant by (ii).  We restrict $F_i(\tEi{L})$ on the open subset $\tEO{V}{\Omega,\le \varepsilon_i(L)}$. Then it is isomorphic to $[\varepsilon_iL]_v(L)|_{\tEO{V}{\Omega,\le \varepsilon_i(L)}}$ and $D$-invariant. Since $L$ is the minimal extension of $L|_{\tEO{V}{\Omega,\le \varepsilon_i(L)}}$, $L$ is $D$-invariant. 
\end{proof}
\begin{rem}
By the result of (iii), we have $a_{L'}(v)=a_{L'}(v^{-1})$ in Theorem \ref{est} (1).
\end{rem}

\begin{lem}\label{lem:div}
For $L \in \tmP{V}{\Omega}$, we have 
\[
F_i^{(a)}(L)=\left[\begin{array}{c}
\varepsilon_i(L)+a \\
a
\end{array}\right]_v(\tFi^aL)+\sum_{L'\colon \varepsilon_i(L')>\varepsilon_i(L)+a}c_{L'}(L')
\]
with $c_{L'} \in \bb{Z}[v,v^{-1}]$.
\end{lem}
\begin{proof}
We shall prove the claim by the induction on $a$. If $a=1$, the claim follows from Theorem \ref{est}. If $a>1$, by the induction hypothesis and Theorem \ref{est}, we have 
\begin{eqnarray*}
F_iF_i^{(a)}(L)&=&\left[\begin{array}{c}
\varepsilon_i(L)+a \\
a
\end{array}\right]_vF_i(\tFi^aL)+\sum_{L'\colon \varepsilon_i(L')>\varepsilon_i(L)+a}c_{L'}F_i(L') \\
&=&[a+1]_v\left(\left[\begin{array}{c}
\varepsilon_i(L)+a+1 \\
a+1
\end{array}\right]_v(\tFi^{a+1}L)+\sum_{L''\colon \varepsilon_i(L'')>\varepsilon_i(L)+a+1}d_{L''}(L'')\right),
\end{eqnarray*}
where $d_{L''} \in \bb{Q}(v)$. Hence
\[
F_i^{(a+1)}L=\left[\begin{array}{c}
\varepsilon_i(L)+a+1 \\
a+1
\end{array}\right]_v(\tFi^{a+1}L)+\sum_{L''\colon \varepsilon_i(L'')>\varepsilon_i(L)+a+1}d_{L''}(L'').
\]
On the other hand, since $F_i^{(a+1)}L=\mb{1}_{\mb{S}_i^{a+1}}*L[d_{a+1}]$ is semisimple, we conclude $d_{L''} \in \bb{Z}[v,v^{-1}]$. 
\end{proof}

\begin{prop}\label{prop-gen}
We have ${}^\theta\!K=\sum{\bb{Z}[v,v^{-1}]F_{i_1}^{(a_1)} \cdots F_{i_k}^{(a_k)}\mb{1}_{\{\pt\}}}$.
\end{prop}
\begin{proof}
For $L \in \tmP{V}{\Omega}$ such that $\wt{\mb{V}} \neq 0$, there exists $i$ such that $\varepsilon_i(L)>0$. We shall prove that $(L)$ is contained in $\sum{\bb{Z}[v,v^{-1}]F_{i_1}^{(a_1)} \cdots F_{i_k}^{(a_k)}\mb{1}_{\{\pt\}}}$ by the induction on $\wt{\mb{V}}$ and the descending induction on $\varepsilon_i(L)$. We have
\[
F_i^{(\varepsilon_i(L))}(\tEi^{\varepsilon_i(L)}{L})=(L)+\sum_{L' \in \tmP{V}{\Omega}\colon \varepsilon_i(L')>\varepsilon_i(L)}c_{L'}(L')
\]
by Lemma \ref{lem:div} and Lemma \ref{lem:cef}. By the induction hypothesis, we have $c_{L'}(L')$ and $\tEi^{\varepsilon_i(L)}{L}$ are contained in $\sum{\bb{Z}[v,v^{-1}]F_{i_1}^{(a_1)} \cdots F_{i_k}^{(a_k)}\mb{1}_{\{\pt\}}}$. Thus $(L) \in \sum{\bb{Z}[v,v^{-1}]F_{i_1}^{(a_1)} \cdots F_{i_k}^{(a_k)}\mb{1}_{\{\pt\}}}$. 
\end{proof}

\subsection{Main Theorem}
Let us recall 
\[
{}^\theta\!K'\colon =\sum_{(\mb{i},\mb{a})}\bb{Z}[v,v^{-1}](\tLL{i}{a}{\Omega})=\sum{\bb{Z}[v,v^{-1}]F_{i_1}^{(a_1)} \cdots F_{i_k}^{(a_k)}\mb{1}_{\{\pt\}}} \subset {}^\theta\!K.
\]
\begin{thm}\label{mt1} \ 
\begin{enumerate}[(i)]
\item ${}^\theta\!K={}^\theta\!K'$. 
\item For $L \in {}^\theta\!\ms{P}_{\mb{V}}$, we define $\wt(L)=-\wt{\mb{V}}$. Then $(\wt,\tEi,\tFi,\varepsilon_i)$ gives a crystal structure on ${}^\theta\!\ms{P}\colon =\sqcup_{\mb{V}}{}^\theta\!\ms{P}_{\mb{V}}$ in the sence of section \ref{sec-cry}. Here $\mb{V}$ runs over all isomorphism classes of $\theta$-symmetric $I$-graded vector spaces.
\item Let $\mathcal{L}$ be the $\mb{A}_0$-submodule $\sum_{(L) \in {}^\theta\!\ms{P}}\mb{A}_0(L)$ of ${}^\theta\!K$. Then $\{(L) \mod {v\mathcal{L}}|L \in {}^\theta\!\ms{P}\}$ gives a crystal basis of ${}^\theta\!K$. Especially, the actions of modified root operators $\tEi$ and $\tFi$ on $\mathcal{L}/v\mathcal{L}$ are compatible with the actions of $\tEi$ and $\tFi$ on ${}^\theta\!\ms{P}$ introduced in Theorem \ref{est}. 
\end{enumerate}
\end{thm}
\begin{proof}
(i) is nothing but Proposition \ref{prop-gen}. \\
(ii) \ By the definition of $\varepsilon_i(L), \tFi$ and $\tEi$, and Lemma \ref{lem:cef}, we conculde that $(\wt,\tEi,\tFi,\varepsilon_i)$ gives a crystal structure on ${}^\theta\!\ms{P}\colon =\sqcup_{\mb{V}}{}^\theta\!\ms{P}_{\mb{V}}$ in the sence of section \ref{sec-cry}(i)-(iv). By the estimates in Theorem \ref{est}, the actions of $E_i$ and $F_i$ on $(L) \ (L \in {}^\theta\!\ms{P})$ satisfy the conditions \ref{c1}-\ref{c6} in section \ref{sec-cry}. Thus we obtain the claim. \\
(iii) follows from Theorem \ref{cricry}.
\end{proof}

\begin{lem}\label{lem:eps2}
We have $\{v \in {}^\theta\!K \ | \ \text{$E_iv=0$ for any $i \in I$}\}=\bb{Z}[v,v^{-1}]\mb{1}_{\{\pt\}}$.
\end{lem}
\begin{proof}
Suppose that $E_i\left(\sum{a_L(L)}\right)=0$ for any $L$. Then $a_L \in v^{c}\bb{Z}[v]$ for some $c$. Put $\widetilde{a_L}=v^{-c}a_L \in \bb{Z}[v]$. By the definition of the modified root operators and Theorem \ref{mt1}(iii), we have $\tEi\left(\sum{\tilde{a_L}(L)}\right)=0$. Specializing $v$ to $0$, we have $\widetilde{a_L}(0)=0$ if $\tEi{L} \neq 0$. But for any $L$ such that $\wt(L) \neq 0$, there exists $i \in I$ such that $\varepsilon_i(L)>0$. Hence we obtain $\widetilde{a_L} \in v\bb{Z}[v]$ and hence $a_L \in v^{c+1}\bb{Z}[v]$. By the induction on $c$, we have $a_L \in v^{c}\bb{Z}[v]$ for any $c$. Thus we conclude $a_L=0$ for $\wt(L) \neq 0$.
\end{proof}

\begin{thm}\label{mt2} \ 
\benu[{\rm(i)}]
\item ${}^\theta\!K \otimes_{\bb{Z}[v,v^{-1}]} \bb{Q}(v) \cong V_\theta(0)$ as a $B_\theta(\mf{g})$-module. The involution induced by the Verdier duality functor coincides with the bar involution on $V_\theta(0)$.
\item $\{(L) \ | \ L \in {}^\theta\!\ms{P}\}$ gives the lower global basis on $V_\theta(0)$.
\eenu
\end{thm}
\begin{proof}
(i) \ By Proposition \ref{prop:action}, to check the defining relations of $B_\theta(\mf{g})$, we only need to prove the $v$-Serre relations. Put
\begin{eqnarray*}
S_e=\sum^b_{k=0}(-1)^kE^{(k)}_iE_jE^{(b-k)}_i, \quad
S_f=\sum^b_{k=0}(-1)^kF^{(k)}_iF_jF_i^{(b-k)}
\end{eqnarray*}
and note that $F_kS_e=S_eF_k$ and $E_kS_f=S_fE_k$ for any $k \in I$. \\
\quad Since $\tKK{\Omega}$ is generated by $F_k^{(n)}$'s from $\phi\colon =\mb{1}_{\{\pt\}}$ and $S_e\phi=0$, we have $S_ev=0$ for any $v \in \tKK{\Omega}$. We show $S_f(L)=0$ for any $L \in \tmP{V}{\Omega}$ by the induction on $\ud{\mb{V}}$. If $\wt(S_f(L)) \neq 0$, we have we have $E_kS_f(L)=S_fE_k(L)=0$ for any $k \in I$ by applying the induction hypothesis to $E_k(L)$. Since $\wt(S_f(L)) \neq 0$, we have $S_f(L)=0$ by Lemma \ref{lem:eps2}. Hence ${}^\theta\!K$ is a $B_\theta(\mf{g})$-module. 
Note that $T_i\mb{1}_{\{\pt\}}=\mb{1}_{\{\pt\}}$ for any $i \in I$. We conclude ${}^\theta\!K \cong V_\theta(0)$ by Lemma \ref{lem:eps2} and the characterization of $V_\theta(0)$ in Proposition \ref{prop:Vtheta}. \\
(ii) \ We already know that $\mathcal{L}=\sum_{L \in {}^\theta\!\ms{P}}{\mb{A}_0(L)}$ is a crystal lattice and $\{(L) \mod{v\mathcal{L}}\}$ is a basis of $\mathcal{L}/v\mathcal{L}$. Note that $\sum_{L \in {}^\theta\!\ms{P}}\bb{Z}[v,v^{-1}](L)$ is stable under the actions of $E_i$'s and $F_i^{(a)}$'s by Lemma \ref{lem:div} and $L$ is $D$-invariant, namely bar-invariant. Moreover $\{(L) \ | \ L \in {}^\theta\!\ms{P}\}$ is a basis of the $\mb{A}_0$-module $\mathcal{L}$ and also a basis of the $\bb{Z}[v,v^{-1}]$-module ${}^\theta\!K$. Hence we conclude that $\{(L) \ | \ L \in {}^\theta\!\ms{P}\}$ gives the lower global basis on $V_\theta(0)$. 
\end{proof}

\begin{cor}
For any Kac-Moody algebra $\mf{g}$ with a symmetric Cartan matrix, the $B_\theta(\mf{g})$-module $V_\theta(0)$ has a crystal basis and a lower global basis, namely Conjecture \ref{conj:crystal} and Conjecture \ref{conj:bal} is true if $\lambda=0$.
\end{cor}

\begin{exa}
Let us consider the case $\mf{g}=\mf{sl}_3$, $I=\{\pm{1}\}$ and $\theta(i)=-i$. Fix a $\theta$-symmetric orientation $-1 \stackrel{\Omega}{\longrightarrow} 1$. For a $\theta$-symmetric $I$-graded vector space $\mb{V}$ such that $\wt(\mb{V})=n(\alpha_{-1}+\alpha_1)$, $\tEO{V}{\Omega}$ is the set of skew symmetric matrix $x$ of size $n$. Its $\tG{V}$-orbits are parametrized by the rank $2r \ (0 \le r \le \lfloor\frac{n}{2}\rfloor)$ of $x$. We denote $\mathcal{O}_r^n$ by the orbit consisting of $n \times n$ skew symmetric matrices $x$ of rank $2r$. Note that any simple local system on each $\tG{V}$-orbit is trivial. Let us denote $\IC_r^n$ by the simple perverse sheaves corresponding to the orbit $\mathcal{O}_r^n$. Note that $\varepsilon_1(\IC_r^n)=n-2r$. \\
\quad Let $\mb{W}$ be a $\theta$-symmetric $I$-graded vector space such that $\wt(\mb{W})=(n-1)(\alpha_{-1}+\alpha_1)$. We consider the diagram:
\[
\xymatrix{
\tEO{W}{\Omega} & \tpE{\Omega}\ar[r]_{p_2}\ar[l]^(.3){p_1} & \tppE{\Omega}\ar[r]_{p_3} & \tEO{V}{\Omega}.
}
\]
Note that the fibers of $p_3$ on $\mathcal{O}_r^n$ is isomorphic to $\mb{P}^{n-1-2r}$. Then
\[
F_1(\IC_r^{n-1})=[n-2r]_v(\IC_r^n)+\sum_{k=0}^{r-1}a_{k,n}(\IC_k^n)
\]
where $a_{k,n} \in v^{2-n+2k}\bb{Z}[v]$. We obtain the crystal graph:
\[
\xymatrix@R=.08em@C=2em{&&&&\IC_0^4\ar@<.2pc>[r]^{1}\ar@<-.2pc>[r]_{-1}&
\IC_0^5\cdots\\
&&\IC_0^2\ar@<.2pc>[r]^{1}\ar@<-.2pc>[r]_{-1}&\IC_0^3
\ar[ru]^{1}\ar[rd]|{-1}\\
\IC_0^0\ar@<.2pc>[r]^(.4){1}\ar@<-.2pc>[r]_(.35){-1}&
\IC_0^1\ar[ru]^{1}\ar[rd]_{-1}&&&\IC_1^4\ar@<.2pc>[r]^{1}\ar@<-.2pc>[r]_{-1}&\IC_1^5\cdots\\
&&\IC_1^2\ar@<.2pc>[r]^{1}\ar@<-.2pc>[r]_{-1}&\IC_1^3\ar[ru]|{1}\ar[rd]_{-1}\\
&&&&\IC_2^4\ar@<.2pc>[r]^{1}\ar@<-.2pc>[r]_{-1}&{\IC_2^5\cdots}
}
\]
Therefore we recover the crystal graph parametrized by "$\theta$-restricted multi-segments" in \cite[Example 4.7 (1)]{EK2}.
\end{exa}

\end{document}